\documentclass[a4paper,11pt,oneside]{article}
\usepackage[utf8]{inputenc}
\usepackage[T1]{fontenc}
\usepackage[english]{babel}
\usepackage[a4paper,top=2.8cm,bottom=3.2cm,left=3.9cm,right=3.3cm]{geometry}
\usepackage[utf8]{inputenc}
\usepackage{graphicx}
\usepackage{amsmath}
\usepackage{amssymb}
\usepackage{amsthm}
\usepackage{color}
\usepackage{version}
\usepackage{mathtools}
\usepackage{tikz}
\usepackage{amsfonts}
\usepackage{bm}
\usepackage{enumitem}
\usepackage{calligra}
\usepackage{calrsfs}
\usepackage{mathabx}
\usepackage{tikz}
\usepackage{caption}
\usepackage{mathrsfs}
\usetikzlibrary{arrows}
\usepackage{mathtools}
\usepackage{bbm}
\theoremstyle{plain}
\newtheorem{thm}{Theorem}
\newtheorem{lem}[thm]{Lemma}
\newtheorem{prop}[thm]{Proposition}
\newtheorem{cor}[thm]{Corollary}
\theoremstyle{definition}

\newtheorem{defn}[thm]{Definition}
\newtheorem{oss}[thm]{Remark}
\newtheorem{ex}[thm]{Example}
\theoremstyle{remark}

\newcommand{\R}{\mathbb{R}}
\newcommand{\C}{\mathbb{C}}
\newcommand{\N}{\mathbb{N}}
\newcommand{\cA}{\mathcal{A}}
\newcommand{\cB}{\mathcal{B}}
\newcommand{\cF}{\mathcal{F}}
\newcommand{\cK}{\mathcal{K}}

\newcommand{\cQ}{\mathcal{Q}}
\newcommand{\cI}{\mathcal{I}}
\newcommand{\sph}[2]{S(#1,#2)}
\newcommand{\ball}[2]{B(#1,#2)}
\newcommand{\sect}[1]{T_{#1}}
\newcommand{\supp}{\mathrm{supp}}

\DeclareMathAlphabet{\mathcal}{OMS}{zplm}{m}{n}

\usepackage{mathtools}
\mathtoolsset{showonlyrefs}

\title{Harmonic Bergman projectors on homogeneous trees}

\author{Filippo De Mari\thanks{Dipartimento di Matematica, Dipartimento di Eccellenza 2023-2027, and MaLGa center, Università di Genova, Via Dodecaneso 35, Genova, Italy, email: demari@dima.unige.it} \and Matteo Monti\thanks{Dipartimento di Scienze Matematiche “Giuseppe Luigi Lagrange”, Politecnico di Torino, Corso Duca degli Abruzzi 24, 10129, Torino, Italy, email: \mbox{matteo.monti@polito.it}, maria.vallarino@polito.it} \and Maria Vallarino\footnotemark[2]}

\begin{document}
	
	\maketitle
	
	\begin{abstract}
		In this paper we investigate some properties of the harmonic Bergman spaces $\mathcal A^p(\sigma)$ on a $q$-homogeneous tree, where $q\geq 2$, $1\leq p<\infty$, and $\sigma$ is a finite measure on the tree with radial decreasing density, hence nondoubling. These spaces were introduced by  J.~Cohen, F.~Colonna, M.~Picardello and D.~Singman. When $p=2$ they are reproducing kernel Hilbert spaces and we compute explicitely their reproducing kernel. We then study the boundedness properties of the Bergman projector on $L^p(\sigma)$ for $1<p<\infty$ and their weak type (1,1) boundedness for radially exponentially decreasing measures on the tree. The weak type (1,1) boundedness is a consequence of the fact that the Bergman kernel satisfies an appropriate integral H\"ormander's condition.  
	\end{abstract}

	%
	%
	\section*{Introduction}
	%
	%
	\let\thefootnote\relax\footnotetext{This work is partially supported by the project "Harmonic analysis on continuous and discrete structures" funded by Compagnia di San Paolo 
		(Cup E13C21000270007). Furthermore, the authors are members of the Gruppo Nazionale per l’Analisi Matematica, la Probabilità e le loro Applicazioni (GNAMPA) of the Istituto Nazionale di Alta Matematica (INdAM). }
	\let\thefootnote\relax\footnotetext{Keywords: Bergman spaces; homogeneous trees; Bergman projectors; Calder\'on-Zygmund theory.}
	\let\thefootnote\relax\footnotetext{Mathematics Subject Classification (2010): 05C05; 46E22; 43A85.}
	
	The main focus of this paper is on the  projectors associated to the  harmonic Bergman  spaces on homogeneous trees introduced in~\cite{ccps}.  The Bergman spaces $\cA^p(\sigma)$, $1\leq p<\infty$,  are in some ways the harmonic analogues of the classical holomorphic Bergman spaces on the hyperbolic disk, whereby $p$-integrability is relative to the reference measure $\sigma$ on the tree, that is a finite measure with radial density with respect to~the counting measure, and where harmonicity is relative to the so-called combinatorial Laplacian. The analogy between the hyperbolic disk and the homogeneous tree inspires many ideas behind our constructions (see~\cite{BW},~\cite{CCEmbedding}).
	
	The space $\cA^2(\sigma)$ is, as expected,  a reproducing kernel Hilbert space (RKHS) and the problem of understanding the associated projectors hinges on the explicit knowledge of the kernel,  an information that we derive by introducing a somewhat canonical basis for $\cA^2(\sigma)$. The core of this contribution is devoted to proving that, for a prototypical class of measures, the extension of the Bergman projector is bounded on $L^p(\sigma)$ if and only if $p>1$, and is of weak type~(1,1). The results are thus almost faithful reformulations of those that hold true for the holomorphic Bergman spaces on the hyperbolic disk (\cite{bbgnpr},~\cite{denghuang},~\cite{forellirudin},~\cite{stein}, and~\cite{zhuoperator}), but many of the key ingredients, first and foremost the explicit determination of the reproducing kernel, call for a rather different approach.
	\bigskip
	
	Let $X$ be a  homogeneous tree. A function on the tree is said to be harmonic if the mean of its values on the neighbors of any vertex coincides with the value at that vertex.
	J.~Cohen, F.~Colonna, M.~Picardello, and D.~Singman introduce  harmonic  Bergman spaces on homogeneous trees in~\cite{ccps}. They consider a family of reference measures which consists of finite measures that are absolutely continuous with respect to the counting measure and whose Radon-Nikodym derivative $\sigma$ is a radial strictly positive decreasing function on $X$. For every $1\leq p<\infty$, the harmonic Bergman space $\cA^p(\sigma)$ is the closed subspace of $L^p(\sigma)$ consisting of harmonic functions. The requirement for the measure $\sigma$ to be finite is suggested by the fact that the only harmonic function which is $p$-integrable with respect to~the counting measure is the null function.

	In the context of the hyperbolic disk, when $p=2$, the weighted Bergman spaces are RKHS, and the holomorphic Bergman kernel is known as well as the properties of the associated projector. Indeed, the extension of the holomorphic Bergman projector to the weighted $L^p$-spaces is bounded if and only if $p>1$, see \cite{forellirudin}, \cite{stein} and~\cite{zhuoperator}. Furthermore, it is of weak type $(1,1)$, see~\cite{bbgnpr} and~\cite{denghuang}.
	In our work, first of all, we show that $\cA^2(\sigma)$ is a RKHS for every reference measure $\sigma$ and we provide an explicit formula for the reproducing kernel $K_\sigma$ in Theorem~\ref{kernelformula}. 
	Since $\cA^2(\sigma)$ is closed in $L^2(\sigma)$, there exists an orthogonal projection $P_\sigma\colon L^2(\sigma)\to \cA^2(\sigma)$. We prove that, for a particular class of reference measures, $P_\sigma$ extends
	to a bounded operator from $L^p(\sigma)$ to $\cA^p(\sigma)$ if and only if $p>1$. Moreover, we show that $P_\sigma$ is of weak type (1,1): to do so we use a Calder\'on-Zygmund decomposition of integrable functions adapted to the measure $\sigma$. Notice that the measure $\sigma$ is not doubling with respect to the standard discrete metric on $X$, but it turns out to be doubling with respect to the so called Gromov metric (see Section \ref{seccz}). Hence a Calder\'on-Zygmund theory in this setting holds, and we show that the Bergman kernel satisfies an integral H\"ormander's condition related to such theory, so that it is of weak type (1,1).

	The measures we focus on are exponentially decreasing radial measures, i.e.~they are exponentially decreasing with respect to the distance from $o$ and can be viewed as natural counterparts of the measures involved in the definition of the standard weighted holomorphic Bergman spaces on the hyperbolic disk. 
	The fact that the extension of the projector to the weighted $L^p$-spaces is bounded if and only if $p>1$ follows from the fact that the projector coincides with a particular Toepliz-type operator (see Section~3.4 in~\cite{zhuoperator}).
	
	In the spirit of the results of~\cite{bekolle} and~\cite{BekolleBonami} on the disk, one could  investigate the boundedness of the Bergman projectors for general reference measures. In~\cite{ccps},~\cite{ccps2},~\cite{ccps3}, the authors introduce and study the optimal measures, a subclass of the reference measures which, roughly speaking, decrease fast as the distance from the origin increases. The exponentially decreasing radial measures are optimal in this sense. The study of the boundedness of the Bergman projector for optimal measures is still work in progress. Another related question is whether the Calder\'on-Zygmund theory that we develop here could be applied to other operators.

	The paper is organized as follows. In the first section we recall the definition of the harmonic Bergman spaces and, for every reference measure, we provide an orthonormal basis of the Hilbert space $\cA^2(\sigma)$. The basis plays a fundamental role in Section~\ref{sec:kernel} in the proof of the two formulae for the kernel of the RKHS $\cA^2(\sigma)$: the first is a recursive formula, while the second is the explicit formula of the kernel given in Theorem~\ref{kernelformula}.  In Section~\ref{sec:bound} we focus on the exponentially decreasing radial measures and state two results characterizing the boundedness of the extension of  a class of Toeplitz-type operators inspired by the operators considered in~\cite{zhuoperator} (see Theorem~\ref{3.11} and Theorem~\ref{3.12}). As a consequence, in Theorem~\ref{theobound} we show that the extension of the harmonic Bergman projector to the weighted $L^p$ spaces is bounded if and only if $p>1$. The last section is devoted to the Calder\'on-Zygmund decomposition of integrable functions (presented in Proposition~\ref{caldzyg}), the formulation of the H\"ormander's type condition, see~\eqref{hormander}, and the weak type (1,1) boundedness of the Bergman projectors is obtained as byproduct.
	
	In what follows, we shall use the symbol $\simeq$ ($\lesssim$, or $\gtrsim$) between two quantities when the left hand side is equal (smaller than or equal to, or greater than or equal to, respectively) to the right hand side up to the multiplication by a (fixed) positive constant. Furthermore we assume the following convention on sums: the sum is null whenever the starting index is greater than the final index. If $Y\subseteq X$, we denote by $\mathbbm{1}_Y$ the characteristic function of $Y$. Finally, we adopt the symbol $\sqcup$ for disjoint unions and $\lfloor x\rfloor$ for the largest integer less than or equal to $x\in\R$.

	\section{Harmonic Bergman spaces}\label{sectone}
	
	\subsection{Preliminaries on homogeneous trees}
	We present some preliminary notions and results on homogeneous trees; for more details we refer to~\cite{cart},~\cite{cms},~\cite{cms2000},~\cite{ftn}. 
	
	A \textit{graph} is a pair $(X,\mathfrak{E})$, where $X$ is the set of \textit{vertices} and $\mathfrak{E}$ is the family of \textit{unoriented edges}, where an edge is a two-element subset of $X$. If two vertices $u,v$ in $X$ are joined by an edge, they are called \textit{adjacent} and this is denoted by $u\sim v$. A \textit{tree} is an undirected, connected, cycle-free graph. A $q$-\textit{homogeneous} tree is a tree in which each vertex has exactly $q+1$ adjacent vertices.  With slight abuse, we refer to the set of vertices $X$ as the tree itself. We fix an origin $o\in X$. 
	
	From now on we consider a $q$-homogeneous tree $X$ with $q\geq 2$.
	Given $u,v\in X$, with $u\neq v$, we denote by $[u,v]$ the unique ordered $t$-uple $(x_0=u,x_1,\dots, x_{t-1}=v)\in X^t$, where $\{x_i,x_{i+1}\}\in\mathfrak{E}$ and all the $x_i$ are distinct. We say that $[u,v]$ is the path starting at $u$ and ending at $v$. With slight abuse of notation, if $[u,v]=(x_0,\dots,x_{t-1})$ we write $x_i\in [u,v]$, $i\in\{0,\dots,t-1\}$. In particular, if $u$ and $v$ are adjacent, both $[u,v],[v,u]\in X^2$ are oriented, unlike the edge $\{u,v\}\in \mathfrak{E}$ which is not. A homogeneous tree $X$ carries a natural distance $d\colon X\times X\to\mathbb{N}$, where for every $u,v\in X$ the distance $d(u,v)$ is the minimal length of a \textit{path} joining $u$ and $v$.
	If $v\in X$, then we denote by $S(v,n)$ and $B(v,n)$  the sphere and the ball centered at $v$ with radius $n\in\mathbb{N}$, respectively, i.e.,
	$$ \sph{v}{n} =\{x\in X\colon d(v,x)= n \}\qquad {\rm{and}}\qquad  \ball{v}{n}=\{x\in X\colon d(v,x)\leq n \}.$$
	It is straightforward to check that 
	\begin{equation}\label{cardsfera}
		\#  \sph{v}{n}=\begin{cases*}
			1, &$n=0$,\\
			(q+1)q^{n-1},&$n\neq 0$.
		\end{cases*} 
	\end{equation}
	We call \textit{norm} of a vertex $v$ in $X$ its distance from $o$, i.e. $|v|=d(o,v)$. We say that a function $f$ on $X$ is \textit{radial} (with respect to~$o$) if its value at a vertex $x\in X$ depends only on $|x|$. If $v\neq o$, then we define the {\it sector of $v$} as the subset
	\[\sect{v}:=\{x\in X\colon [o,v]\subseteq [o,x] \},\]
	and we adopt the convention $\sect{o}=X$. Moreover, we call {\it successors} of $v$ the elements of the set
	$s(v)=\{u\in X\colon u\sim v,|u|=|v|+1\}.$ Evidently,
	\[\#  s(v)=\begin{cases}
		q,&\text{if }v\neq o;\\
		q+1,&\text{if }v=o.
	\end{cases} \] 
	For every $v\neq o$ we call predecessor of $v$ and denote by $p(v)$ the only neighbor of $v$ which is not a successor of $v$; it follows that $|p(v)|=|v|-1$. The vertex $o$ is the only one having no predecessors, and $s(o)=\sph{o}{1}$. This defines the predecessor function $p\colon X\setminus\{o\}\to X$, and, for every positive integer $\ell$, its $\ell$-fold composition $p^\ell\colon X\setminus\ball{o}{\ell-1}\to X$ is the $\ell$-th predecessor function.

	\subsection{Harmonic functions and harmonic Bergman spaces}\label{sub:harm}
	
	
	\begin{defn}
		Let $f$ be a complex valued function on $X$. The \textit{combinatorial Laplacian} of $f$ is defined by
		\[Lf(v):=f(v)-\frac{1}{q+1}\sum_{u\sim v}f(u),\qquad  v\in X.\] 
		We say that $f$ is \textit{harmonic} on $Y\subseteq X$ if $Lf=0$ on $Y$. Equivalently, $f$ is harmonic on $Y$ if 
		\begin{equation}\label{harmonicity}
			f(v)=\frac{1}{q+1}\sum_{u\sim v}f(u), \qquad v\in Y.
		\end{equation} 
		We shall call a function harmonic if it is harmonic on $X$. 
	\end{defn}

	It is easy to prove that a function is harmonic if and only if for every $v\in X$ and $n\in \N$, we have
	\begin{equation}\label{harmoball}
		f(v)=\frac{1}{\# \sph{v}{n}}\sum_{d(v,x)=n}f(x).
	\end{equation}
	
	The harmonicity property \eqref{harmonicity} implies a certain rigidity for the function. In particular, the value of a harmonic function at a vertex $y\in X$ “propagates'' to every layer of the sector $\sect{y}$, as showed in the following proposition, which is a modified version of~\cite[Lemma 4.1]{ccps}. In that lemma, the authors show that a function which is harmonic and radial on a sector $\sect{y}$, $y\in X\setminus\{o\}$, is completely described by its values at $y$ and $p(y)$. We consider a harmonic function on the sector $\sect{y}$, removing the radiality assumption, and we formulate a result for its average on $\sph{o}{n}\cap\sect{y}$, $n\geq |y|$. We omit the proof since it is an easy adaptation of the proof of~\cite[Lemma~4.1]{ccps}.
	\begin{prop}\label{prop1}
		Let $y\in X\setminus \{o\}$. If $f\colon X\rightarrow\mathbb{C}$ is harmonic on $\sect{y}$, then for every $n\in \mathbb{N}$, $n\geq|y|$, we have
		\begin{equation}\label{harmTybis}
			\sum_{\substack{|x|=n\\x\in \sect{y}}}f(x)=\left(\sum_{j=0}^{n-|y|}q^j\right)f(y)-\left(\sum_{j=0}^{n-|y|-1}q^j\right)f(p(y)).
		\end{equation}	
		Furthermore, if $f\colon X\to \C$ is radial on $T_y$ and satisfies \eqref{harmTybis} for every $n\geq |y|$, then $f$ is harmonic on $\sect{y}$.
	\end{prop}

	We introduce a technique which allows to extend a function which is harmonic on a ball  centered in $o$ to a function harmonic on $X$. Let $n\in\N$ and $g$ be a function on $X$ which is harmonic on $\ball{o}{n}$. It is easy to see that there are infinitely many ways to extend $g$ to a harmonic function on $X$ which coincides with $g$ on $\ball{o}{n+1}$.
	As we see next, there is however a unique harmonic function $g_n^H$ on $X$ which is radial when restricted on $\sect{y}$ for every $y\in\sph{o}{n+1}$.

	Let $x\in X\setminus\ball{o}{n}$. There exists a unique $y\in\sph{o}{n+1}$ such that $x\in\sect{y}$, and $y=p^{|x|-n-1}(x)$ (where $p^0={\rm id}_X$). Since we aim to construct $g^H_{n}$ radial and harmonic on $\sect{y}$, by Proposition~\ref{prop1} we have that

	\begin{align*}
		g_{n}^H(x)&=
		\frac{1}{\#  \sph{o}{|x|}\cap\sect{y}}\sum_{\substack{|z|=|x|,\\ z\in\sect{y}}}g_n^H(z)\\
		&=
		q^{|y|-|x|}\left[\left(\sum_{j=0}^{|x|-|y|}q^j\right)g(y)-\left(\sum_{j=0}^{|x|-|y|-1}q^j\right)g(p(y))\right]\\
		&=
		q^{n+1-|x|}\left[\left(\sum_{j=0}^{|x|-n-1}q^j\right)g(p^{|x|-n-1}(x))-\left(\sum_{j=0}^{|x|-n-2}q^j\right)g(p^{|x|-n}(x))\right]\\
		&=
		\left(\sum_{j=0}^{|x|-n-1}q^{-j}\right)g(p^{|x|-n-1}(x))-\left(\sum_{j=1}^{|x|-n-1}q^{-j}\right)g(p^{|x|-n}(x)).
	\end{align*}
	For simplicity we introduce the notation 
	\[a_n=\sum_{j=0}^{n}q^{-j}=\frac{q-q^{-n}}{q-1},\qquad n\in\N,\]
	and we set $a_{-1}=0$.
	Hence 
	\begin{equation}\label{gextension}
		g^H_{n}(x)=\begin{dcases}
			g(x),&|x|\leq n,\\
			a_{|x|-n-1}g(p^{|x|-n-1}(x))-\left(a_{|x|-n-1}-1\right)g(p^{|x|-n}(x)),&|x|>n.
		\end{dcases} 
	\end{equation}
	The function $g^H_{n}$ defined above is harmonic on $X$ by Proposition~\ref{prop1} and because 
	\[ X=\ball{o}{n}\cup\bigcup_{y\in\sph{o}{n+1}}\sect{y}.\]
	Observe that  $g^H_{n}$ is indeed harmonic on $\ball{o}{n}$ because $a_0=1$ and $a_{-1}=0$, yield $g^H_{n}=g$ on $\ball{o}{n+1}$, and not only on $\ball{o}{n}$. Furthermore,
	the extension $g^H_{n}$ is radial on every sector \lq\lq starting'' from a point in $\sph{o}{n+1}$ by construction.

	\begin{center}
		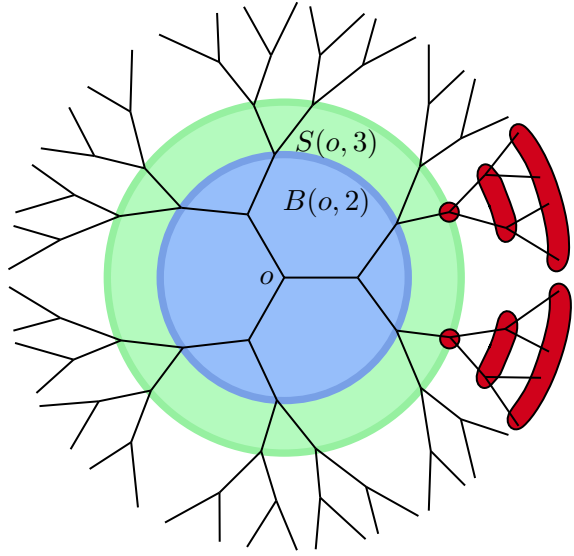
\begin{figure}[h]
			\centering

			
			\tikzset {_bhbefrn48/.code = {\pgfsetadditionalshadetransform{ \pgftransformshift{\pgfpoint{0 bp } { 0 bp }  }  \pgftransformrotate{0 }  \pgftransformscale{2 }  }}}
			\pgfdeclarehorizontalshading{_0stppf1v2}{150bp}{rgb(0bp)=(0.95,0.77,0.77);
				rgb(37.5bp)=(0.95,0.77,0.77);
				rgb(62.5bp)=(0.82,0.01,0.11);
				rgb(100bp)=(0.82,0.01,0.11)}
			\tikzset{_tizpfi7l4/.code = {\pgfsetadditionalshadetransform{\pgftransformshift{\pgfpoint{0 bp } { 0 bp }  }  \pgftransformrotate{0 }  \pgftransformscale{2 } }}}
			\pgfdeclarehorizontalshading{_xj6ykzquq} {150bp} {color(0bp)=(transparent!87);
				color(37.5bp)=(transparent!87);
				color(62.5bp)=(transparent!81);
				color(100bp)=(transparent!81) } 
			\pgfdeclarefading{_pqjnbhbmi}{\tikz \fill[shading=_xj6ykzquq,_tizpfi7l4] (0,0) rectangle (50bp,50bp); } 
			
			
			\tikzset {_9mzvgomtb/.code = {\pgfsetadditionalshadetransform{ \pgftransformshift{\pgfpoint{0 bp } { 0 bp }  }  \pgftransformrotate{0 }  \pgftransformscale{2 }  }}}
			\pgfdeclarehorizontalshading{_2oobo7vya}{150bp}{rgb(0bp)=(0.95,0.77,0.77);
				rgb(37.5bp)=(0.95,0.77,0.77);
				rgb(62.5bp)=(0.82,0.01,0.11);
				rgb(100bp)=(0.82,0.01,0.11)}
			\tikzset{_4j8c4uvg0/.code = {\pgfsetadditionalshadetransform{\pgftransformshift{\pgfpoint{0 bp } { 0 bp }  }  \pgftransformrotate{0 }  \pgftransformscale{2 } }}}
			\pgfdeclarehorizontalshading{_k7n897beo} {150bp} {color(0bp)=(transparent!87);
				color(37.5bp)=(transparent!87);
				color(62.5bp)=(transparent!81);
				color(100bp)=(transparent!81) } 
			\pgfdeclarefading{_nmuib15vf}{\tikz \fill[shading=_k7n897beo,_4j8c4uvg0] (0,0) rectangle (50bp,50bp); } 
			
			
			\tikzset {_wm6utsgwi/.code = {\pgfsetadditionalshadetransform{ \pgftransformshift{\pgfpoint{0 bp } { 0 bp }  }  \pgftransformscale{1 }  }}}
			\pgfdeclareradialshading{_aeg1z864q}{\pgfpoint{0bp}{0bp}}{rgb(0bp)=(0.72,0.91,0.53);
				rgb(4.196428571428571bp)=(0.72,0.91,0.53);
				rgb(8.035714285714286bp)=(0.72,0.91,0.53);
				rgb(25bp)=(0.72,0.91,0.53);
				rgb(400bp)=(0.72,0.91,0.53)}
			\tikzset{_7hsyc9jl9/.code = {\pgfsetadditionalshadetransform{\pgftransformshift{\pgfpoint{0 bp } { 0 bp }  }  \pgftransformscale{1 } }}}
			\pgfdeclareradialshading{_webrkgomf} { \pgfpoint{0bp} {0bp}} {color(0bp)=(transparent!69);
				color(4.196428571428571bp)=(transparent!69);
				color(8.035714285714286bp)=(transparent!58);
				color(25bp)=(transparent!36);
				color(400bp)=(transparent!36)} 
			\pgfdeclarefading{_j3w0gmo3z}{\tikz \fill[shading=_webrkgomf,_7hsyc9jl9] (0,0) rectangle (50bp,50bp); } 
			
			
			\tikzset {_vw9z38ine/.code = {\pgfsetadditionalshadetransform{ \pgftransformshift{\pgfpoint{0 bp } { 0 bp }  }  \pgftransformscale{1 }  }}}
			\pgfdeclareradialshading{_hdp2ch4yu}{\pgfpoint{0bp}{0bp}}{rgb(0bp)=(0.29,0.56,0.89);
				rgb(0bp)=(0.29,0.56,0.89);
				rgb(25bp)=(0.29,0.56,0.89);
				rgb(400bp)=(0.29,0.56,0.89)}
			\tikzset{_rdahlcf6z/.code = {\pgfsetadditionalshadetransform{\pgftransformshift{\pgfpoint{0 bp } { 0 bp }  }  \pgftransformscale{1 } }}}
			\pgfdeclareradialshading{_hed0nxjd0} { \pgfpoint{0bp} {0bp}} {color(0bp)=(transparent!73);
				color(0bp)=(transparent!73);
				color(25bp)=(transparent!50);
				color(400bp)=(transparent!50)} 
			\pgfdeclarefading{_ru7lgr3ux}{\tikz \fill[shading=_hed0nxjd0,_rdahlcf6z] (0,0) rectangle (50bp,50bp); } 
			\tikzset{every picture/.style={line width=0.75pt}} 
			
			\begin{tikzpicture}[scale=1.25,x=0.75pt,y=0.75pt,yscale=-1,xscale=1]

				\fill
				[color={rgb, 255:red, 180; green, 250; blue, 190 }  ,draw opacity=0.4 ]
				(239.88,150) .. controls (239.88,110.99) and (271.49,79.38) .. (310.5,79.38) .. controls (349.51,79.38) and (381.13,110.99) .. (381.13,150) .. controls (381.13,189.01) and (349.51,220.63) .. (310.5,220.63) .. controls (271.49,220.63) and (239.88,189.01) .. (239.88,150) -- cycle ; 
				\draw [color={rgb, 255:red, 150; green, 240; blue, 160 }  ,draw opacity=1 ][line width=2.75]  (239.88,150) .. controls (239.88,110.99) and (271.49,79.38) .. (310.5,79.38) .. controls (349.51,79.38) and (381.13,110.99) .. (381.13,150) .. controls (381.13,189.01) and (349.51,220.63) .. (310.5,220.63) .. controls (271.49,220.63) and (239.88,189.01) .. (239.88,150) -- cycle ; 
				
				\fill  
				[color={rgb, 255:red, 150; green, 190; blue, 250 }  ,draw opacity=1 ] (261,150) .. controls (261,122.66) and (283.16,100.5) .. (310.5,100.5) .. controls (337.84,100.5) and (360,122.66) .. (360,150) .. controls (360,177.34) and (337.84,199.5) .. (310.5,199.5) .. controls (283.16,199.5) and (261,177.34) .. (261,150) -- cycle ; 
				\draw  [color={rgb, 255:red, 120; green, 160; blue, 230 }  ,draw opacity=1 ][line width=2.75]  (261,150) .. controls (261,122.66) and (283.16,100.5) .. (310.5,100.5) .. controls (337.84,100.5) and (360,122.66) .. (360,150) .. controls (360,177.34) and (337.84,199.5) .. (310.5,199.5) .. controls (283.16,199.5) and (261,177.34) .. (261,150) -- cycle ; 

				\draw  [draw opacity=0][fill={rgb, 255:red, 208; green, 2; blue, 27 }  ,fill opacity=0.26 ] (405.97,209.77) .. controls (410.22,206.04) and (414.85,197.87) .. (417.77,190.97) .. controls (420.69,184.07) and (425.25,170.82) .. (424.89,160.21) .. controls (424.53,149.59) and (415.7,150.78) .. (416.43,158.14) .. controls (417.16,165.51) and (413.91,176.65) .. (410.25,186.69) .. controls (406.6,196.73) and (403.45,199.45) .. (400.62,204.02) .. controls (397.78,208.59) and (401.72,213.51) .. (405.97,209.77) -- cycle ;
				\draw  [draw opacity=0][fill={rgb, 255:red, 208; green, 2; blue, 27 }  ,fill opacity=0.26 ] (393.47,192.75) .. controls (396.96,188.54) and (395.98,190.84) .. (398.67,184.27) .. controls (401.35,177.7) and (404.21,172.46) .. (403.38,167.76) .. controls (402.55,163.06) and (397.28,161.2) .. (395.71,167.82) .. controls (394.13,174.43) and (391.24,180.66) .. (392.69,177.38) .. controls (394.15,174.1) and (391.61,179.68) .. (388.37,186.23) .. controls (385.13,192.78) and (389.98,196.96) .. (393.47,192.75) -- cycle ;
				\draw  [draw opacity=0][fill={rgb, 255:red, 208; green, 2; blue, 27 }  ,fill opacity=0.26 ] (372.73,174.54) .. controls (372.65,176.66) and (374.3,178.44) .. (376.42,178.51) .. controls (378.53,178.59) and (380.31,176.94) .. (380.39,174.82) .. controls (380.46,172.71) and (378.81,170.93) .. (376.7,170.85) .. controls (374.58,170.78) and (372.8,172.43) .. (372.73,174.54) -- cycle ;
				
				\draw  [draw opacity=0][fill={rgb, 255:red, 208; green, 2; blue, 27 }  ,fill opacity=0.26 ] (407,89.67) .. controls (411.11,93.56) and (415.44,101.89) .. (418.11,108.89) .. controls (420.78,115.89) and (424.86,129.29) .. (424.11,139.89) .. controls (423.37,150.49) and (414.58,148.98) .. (415.58,141.64) .. controls (416.58,134.31) and (413.74,123.06) .. (410.44,112.89) .. controls (407.15,102.72) and (404.11,99.89) .. (401.44,95.22) .. controls (398.78,90.56) and (402.89,85.78) .. (407,89.67) -- cycle ;
				\draw  [draw opacity=0][fill={rgb, 255:red, 208; green, 2; blue, 27 }  ,fill opacity=0.26 ] (393.89,106.22) .. controls (397.22,110.56) and (396.33,108.22) .. (398.78,114.89) .. controls (401.22,121.56) and (403.89,126.89) .. (402.89,131.56) .. controls (401.89,136.22) and (396.56,137.89) .. (395.22,131.22) .. controls (393.89,124.56) and (391.22,118.22) .. (392.56,121.56) .. controls (393.89,124.89) and (391.56,119.22) .. (388.56,112.56) .. controls (385.56,105.89) and (390.56,101.89) .. (393.89,106.22) -- cycle ;
				\draw  [draw opacity=0][fill={rgb, 255:red, 208; green, 2; blue, 27 }  ,fill opacity=0.26 ] (372.5,123.67) .. controls (372.5,121.55) and (374.22,119.83) .. (376.33,119.83) .. controls (378.45,119.83) and (380.17,121.55) .. (380.17,123.67) .. controls (380.17,125.78) and (378.45,127.5) .. (376.33,127.5) .. controls (374.22,127.5) and (372.5,125.78) .. (372.5,123.67) -- cycle ;
				
				\draw    (376.33,123.67) -- (390.67,108.67) ;
				\draw    (399,130) -- (420.2,143) ;
				\draw    (376.33,123.67) -- (399,130) ;
				\draw    (390.67,108.67) -- (404,91.67) ;
				\draw    (310.5,150) -- (339.81,150) ;
				\draw    (339.81,150) -- (355.33,171.33) ;
				\draw    (339.81,150) -- (355.33,128.33) ;
				\draw    (355.33,171.33) -- (365,194.33) ;
				\draw    (355.33,171.33) -- (376.67,174) ;
				\draw    (355.33,128.33) -- (376.33,123.67) ;
				\draw    (355.33,128.33) -- (364.67,105.33) ;
				\draw    (365,194.33) -- (369,219.8) ;
				\draw    (365,194.33) -- (381,207) ;
				\draw    (376.67,174) -- (390.67,190) ;
				\draw    (376.67,174) -- (398.67,170.67) ;
				\draw    (369,219.8) -- (368.2,243.8) ;
				\draw    (369,219.8) -- (382.67,233.67) ;
				\draw    (381,207) -- (386.67,230.67) ;
				\draw    (381,207) -- (400,213.33) ;
				\draw    (390.67,190) -- (404,209.67) ;
				\draw    (390.67,190) -- (413,190.33) ;
				\draw    (398.67,170.67) -- (420.2,155.4) ;
				\draw    (398.67,170.67) -- (416.2,179.8) ;
				\draw    (364.67,105.33) -- (368.2,80.2) ;
				\draw    (364.67,105.33) -- (381,94) ;
				\draw    (368.2,80.2) -- (368.2,56.2) ;
				\draw    (368.2,80.2) -- (382.33,66.67) ;
				\draw    (381,94) -- (386.67,70) ;
				\draw    (381,94) -- (400,85.33) ;
				\draw    (390.67,108.67) -- (412.67,109.67) ;
				\draw    (399,130) -- (416.2,120.2) ;
				\draw    (310.49,150) -- (295.89,124.58) ;
				\draw    (295.89,124.58) -- (306.66,100.5) ;
				\draw    (295.89,124.58) -- (269.38,121.91);
				\draw    (306.66,100.5) -- (321.79,80.66) ;
				\draw    (306.66,100.5) -- (298.35,80.67) ;
				\draw    (269.38,121.91) -- (254.87,106.03) ;
				\draw    (269.38,121.91) -- (244.78,125.27) ;
				\draw    (321.79,80.66) -- (341.89,64.51) ;
				\draw    (321.79,80.66) -- (324.81,60.48) ;
				\draw    (298.35,80.67) -- (305.25,60.56) ;
				\draw    (298.35,80.67) -- (284.5,63.25) ;
				\draw    (341.89,64.51) -- (363.1,53.25) ;
				\draw    (341.89,64.51) -- (347.1,45.75) ;
				\draw    (324.81,60.48) -- (342.51,43.78) ;
				\draw    (324.81,60.48) -- (320.84,40.85) ;
				\draw    (305.25,60.56) -- (315.67,39.2) ;
				\draw    (305.25,60.56) -- (294.42,41.03) ;
				\draw    (284.5,63.25) -- (260.54,52.18) ;
				\draw    (284.5,63.25) -- (283.69,43.5) ;
				\draw    (244.78,125.27) -- (221.23,134.73) ;
				\draw    (244.78,125.27) -- (226.82,116.76) ;
				\draw    (254.87,106.03) -- (234.73,101.07) ;
				\draw    (254.87,106.03) -- (249.07,83.22) ;
				\draw    (221.23,134.73) -- (200.42,146.68) ;
				\draw    (221.23,134.73) -- (202.45,129.21) ;
				\draw    (226.82,116.76) -- (203.19,123.79) ;
				\draw    (226.82,116.76) -- (209.84,104.6) ;
				\draw    (234.73,101.07) -- (213.34,97.97) ;
				\draw    (234.73,101.07) -- (224.64,81.49) ;
				\draw    (249.07,83.22) -- (232.01,73.18) ;
				\draw    (249.07,83.22) -- (249.79,58.36) ;
				\draw    (310.72,149.61) -- (296.32,175.14) ;
				\draw    (296.32,175.14) -- (270.11,178.18) ;
				\draw    (296.32,175.14) -- (307.56,199.3) ;
				\draw    (270.11,178.18) -- (245.33,175.29) ;
				\draw    (270.11,178.18) -- (257.3,195.45) ;
				\draw    (307.56,199.3) -- (301.31,219.89) ;
				\draw    (307.56,199.3) -- (323.01,218.73) ;
				\draw    (245.33,175.29) -- (221.18,166.26) ;
				\draw    (245.33,175.29) -- (226.43,183.01) ;
				\draw    (257.3,195.45) -- (236.49,199.78) ;
				\draw    (257.3,195.45) -- (249.4,216.24) ;
				\draw    (221.18,166.26) -- (200.67,153.77) ;
				\draw    (221.18,166.26) -- (202.39,171.35) ;
				\draw    (226.43,183.01) -- (203.04,176.31) ;
				\draw    (226.43,183.01) -- (211.58,196.44) ;
				\draw    (236.49,199.78) -- (212.81,201.73) ;
				\draw    (236.49,199.78) -- (225.23,219.07) ;
				\draw    (249.4,216.24) -- (252.11,242.5) ;
				\draw    (249.4,216.24) -- (232.83,227.03) ;
				\draw    (323.01,218.73) -- (343.16,234.16) ;
				\draw    (323.01,218.73) -- (324.85,238.53) ;
				\draw    (301.31,219.89) -- (307.33,239.74) ;
				\draw    (301.31,219.89) -- (284.65,236.52) ;
				\draw    (343.16,234.16) -- (364.06,245.95) ;
				\draw    (343.16,234.16) -- (348,253.12) ;
				\draw    (324.85,238.53) -- (342.97,255.26) ;
				\draw    (324.85,238.53) -- (323.06,259.34) ;
				\draw    (307.33,239.74) -- (315.58,259.71) ;
				\draw    (307.33,239.74) -- (295.65,258.41) ;
				\draw    (284.65,236.52) -- (284.74,256.31) ;
				\draw    (284.65,236.52) -- (262.91,248.59) ;
				
				\node[anchor=east]  at (310.49,150){$o$};
				\node  at (327.49,120){$\ball{o}{2}$};
				\node  at (331.49,96){$\sph{o}{3}$};	
				
			\end{tikzpicture}

			\caption{
				The function $g$ is harmonic on $\ball{o}{2}$, that is the set of vertices in the blue area. The function $g^H_{2}$ is obtained by extending the values of $g$ in $\sph{o}{3}$ (the green area) along sectors in such a way that $g^H_{2}$ is harmonic on $X$ and constant on the vertices lying on the same red arc, that is on the \lq\lq layers'' of the sectors.}
		\end{figure}
	\end{center}

	\subsection{Harmonic Bergman spaces}\label{sub:berg}
	Homogeneous trees are classically endowed with the counting measure. The main feature of such measure is the invariance under the group of isometries of the tree. 
	When studying spaces of harmonic functions, this measure is however inadequate because the only harmonic function that is $p$-summable, $1\leq p<\infty$, with respect to~the counting measure is the null function, as we show in the following statement. 
	\begin{prop}\label{corollharmnull}
		If $f$ is a harmonic function in $L^p(X)$, $1\leq p<\infty$, then $f$ is the null function.
	\end{prop}
	\begin{proof}
		Suppose that $f$ is harmonic. We have that
		\begin{align*}
			\sum_{x\in X}|f(x)|^p=&\sum_{n=0}^{+\infty}\sum_{|x|=n}|f(x)|^p\\
			=&\frac{1}{(q+1)^p}\sum_{n=0}^{+\infty}\sum_{|x|=n}\left|\sum_{y\sim x}f(y)\right|^p\\
			\leq&\frac{(q+1)^{p-1}}{(q+1)^p}\sum_{n=0}^{+\infty}\sum_{|x|=n}\sum_{y\sim x}|f(y)|^p\\
			=&\frac{1}{q+1}(q+1)\|f\|_{L^p(X)}^p=\|f\|_{L^p(X)}^p<+\infty,
		\end{align*}
		since every vertex is neighbor of exactly $q+1$ other vertices.
		Hence the unique inequality in the computation above is an equality, so that  
		\[(q+1)^{p-1}\sum_{y\sim x}|f(y)|^p= \left|\sum_{y\sim x}f(y)\right|^p=(q+1)^p|f(x)|^p,\]
		which means that $|f|^p$ is harmonic, too. If $f$ is not the null function, then there exists $v\in X$ such that $f(v)\neq 0$. Hence by~\eqref{harmoball}, we have
		\[\sum_{x\in X}|f(x)|^p=\sum_{n=0}^{+\infty}\sum_{d(v,x)=n}|f(x)|^p=|f(v)|^p\sum_{n=0}^{+\infty}\#  \sph{v}{n}=+\infty,\]
		which is a contradiction. Hence $f=0$.
	\end{proof}

	Since we are interested in Bergman spaces of harmonic functions, the previous proposition leads to consider finite measures on $X$.
	In~\cite{ccps}, the authors introduce
	harmonic Bergman spaces with respect to the following class of measures.
	\begin{defn}
		A {\it reference measure} on $X$ is a finite measure that is absolutely continuous with respect to~the counting measure and whose Radon-Nikodym derivative $\sigma$ is a radial strictly positive decreasing function on $X$. With slight abuse of notation we denote by $\sigma$ the reference measure, too.
		Given a reference measure $\sigma$ on $X$ for every $p\in [1,\infty)$ the \textit{Bergman space} $\mathcal{A}^p(\sigma)$ is the space of harmonic functions on $X$ such that
		\[\|f\|_{\mathcal{A}^p(\sigma)}^p:=\sum_{x\in X}|f(x)|^p\sigma(x)<+\infty. \]
	\end{defn}
		Every Bergman space $\mathcal A^p(\sigma)$ is a Banach space and when $p=2$, it is a Hilbert space with the scalar product
			\begin{equation}\label{scalprod}
					\langle f,g\rangle_{\mathcal A^2(\sigma)}:=\sum_{x\in X}f(x)\overline{g(x)}\sigma(x),\qquad f,g\in \mathcal A^2(\sigma). 
			\end{equation}
		
	If $\sigma$ is a reference measure on $X$, and if we denote by $\sigma_n$ the value of $\sigma$ on the sphere $S(0,n)$, then by~\eqref{cardsfera} the total mass of $\sigma$ is denoted by $B_\sigma$ and its value is
	\[B_\sigma=\sigma_0+\frac{q+1}{q}\sum_{n=1}^{+\infty}\sigma_nq^n<+\infty.\]
	\begin{ex}\label{ex}
		Let $\alpha>1$. Interesting examples of reference measures are the exponentially decreasing radial measures, consisting of the measures having density 
		\[\mu_\alpha(x)=q^{-\alpha|x|},\qquad x\in X.\]
		Indeed, $\mu_\alpha$ is radial, positive and decreasing.
		Furthermore, we write $B_\alpha$ in place of $B_{\mu_\alpha}$, namely
		\begin{equation}\label{misuratotalealpha}
			B_\alpha=1+\frac{q+1}{q}\sum_{n=1}^{+\infty}q^{(1-\alpha)n}\\
			=1+\frac{q+1}{q}\frac{q^{1-\alpha}}{1-q^{1-\alpha}}=\frac{q^\alpha+1}{q^\alpha-q}<+\infty.
		\end{equation}
	\end{ex}

	\begin{prop}\label{nondoubling}
		For every reference measure $\sigma$ the measure metric space $(X,d,\sigma)$ is nondoubling.
	\end{prop}
	\begin{proof}
		Let $\sigma$ be a reference measure. For every $n\in\N$, let $v_n\in X$ be such that $|v_n|=2n$. Then $\max\{\sigma(x)\colon x\in\ball{v_n}{n}\}=\sigma_n$ and so
		\[\sigma(\ball{v_n}{n})=\sum_{x\in\ball{v_n}{n}}\sigma(x)\leq \sigma_n|\ball{v_n}{n}|\lesssim q^{n}\sigma_n.\]
		On the other side, since $o\in\ball{v_n}{2n}$, we have
		\[\frac{\sigma(\ball{v_n}{2n})}{\sigma(\ball{v_n}{n})}\gtrsim \frac{\sigma(o)}{q^n\sigma_n}\xrightarrow{n\to\infty}\infty, \]
		by the finiteness of $\sigma$. This concludes the proof.
	\end{proof}

	Given a reference measure $\sigma$, we introduce the decreasing sequence $(b_{n})_{n\in\N}$ which collects some important information on $\sigma$. For every $n\in\N$, we define
	\begin{equation}\label{bt}\begin{split}
			b_n=b_n(\sigma)=\sum_{m=n+1}^{+\infty}\left[\sigma_{ m}a_{m-n-1}\left(\sum_{k=0}^{m-n-1}q^{k}\right)\right].
		\end{split}
	\end{equation}
	The sums are finite because $\sigma$ is a finite measure on $X$.
	We shall use the notation $b_n$ instead of $b_n(\sigma)$ whenever the measure is clear from the context.
	
	The next lemma is a technical result that is very useful in what follows. 
	\begin{lem}\label{lemma}
		Let $n\in\N$ and $g$ be a function on $X$ which is harmonic and vanishes on $\ball{o}{n}$. Then there exists a constant $b_{n}'>0$ such that for every $f\in\mathcal{A}^2(\sigma)$ 
		\[\langle f,g^H_{n} \rangle_{\mathcal{A}^2(\sigma)}=
		\sum_{|y|=n+1}\left(b_n f(y)-b'_n f(p(y))\right)\overline{g(y)}, 
		\]
		where $b_n$ is defined in~\eqref{bt} and $\langle\cdot,\cdot\rangle_{\cA^2(\sigma)}$ in~\eqref{scalprod}.
		%
	\end{lem}
	\begin{oss}
		The constant $b_n'$ has a structure similar to that of $b_n$, as can be seen in the proof below, but we are not interested in it. 
	\end{oss}
	\begin{proof} 
		Observe that, from the fact that $g|_{\ball{o}{n}}=0$ and~\eqref{gextension}, for every $x\in X$ with $|x|>n$ we have $g^H_n(x)=a_{|x|-n-1}g(p^{|x|-n-1}(x))$.
		Take $f\in \mathcal{A}^2(\sigma)$. Then, by applying Proposition~\ref{prop1} to $f$, we have
		\begin{align*}
			\langle f,&g^H_{n} \rangle_{\mathcal{A}^2(\sigma)}=\sum_{m=n+1}^{+\infty}\sigma_m\sum_{|x|=m}f(x)\overline{g^H_{n}(x)}\\
			&=\sum_{m=n+1}^{+\infty}\sigma_m\sum_{|y|=n+1}\sum_{\substack{|x|=m\\x\in\sect{y}}}f(x)a_{|x|-n-1}\overline{g(p^{|x|-n-1}(x))}\\
			&=\sum_{|y|=n+1}\overline{g(y)}\sum_{m=n+1}^{+\infty}\sigma_ma_{m-n-1}\sum_{\substack{|x|=m\\x\in\sect{y}}}f(x)\\
			&=\sum_{|y|=n+1}\overline{g(y)}\sum_{m=n+1}^{+\infty}\sigma_ma_{m-n-1}\left[\left(\sum_{k=0}^{m-n-1}q^{k}\right)f(y)-\left(\sum_{k=0}^{m-n-2}q^{k}\right)f(p(y))\right]\\
			&=\sum_{|y|=n+1}\left(b_n f(y)-b'_n f(p(y))\right)\overline{g(y)},
		\end{align*}
		as required.
	\end{proof}

	\subsection{A canonical orthonormal basis of $\mathcal{A}^2(\sigma)$}\label{sub:basis}
	The goal of this section is the construction of an orthonormal basis of the space $\mathcal{A}^2(\sigma)$.

	Let us consider the linear spaces 
	\begin{equation*}
		W_v:=\big\{\varphi\colon s(v)\to\mathbb{C}\colon \sum_{z\in s(v)}\varphi(z)=0 \big\}\simeq\C^{|s(v)|-1}=
		\begin{cases}
			\mathbb{C}^{q}, &v=o,\\
			\mathbb{C}^{q-1},& v\in X\setminus\{o\}.
		\end{cases} 
	\end{equation*}
	For every $v\in X$ we set
	$I_v=\{1,\dots,|s(v)|-1\}.$
	For every $v\in X$ we fix an orthonormal basis $\{e_{v,j}\}_{j\in I_v}$ of $W_v$ with respect to to the scalar product 
	\[\langle \varphi,\psi\rangle_{W_v}=\sum_{y\in s(v)}\varphi(y)\overline{\psi(y)}.\]

	Let $v\in X$ and $j\in I_v$. We consider the extension by zero to all of $X$ of $e_{v,j}$, namely,
	\[E_{v,j}(x)=\begin{cases}
		e_{v,j}(x),&x\in s(v);\\
		0,&x\not\in s(v).
	\end{cases} \]
	It is immediate to see that $E_{v,j}$ is harmonic on $B(o,|v|)$ and vanishes on $B(o,|v|)$. 
	We denote the harmonic extension of $E_{v,j}$ by $ f_{v,j}=(E_{v,j})^H_{|v|}$, namely
	\begin{equation}\label{fvj}
		f_{v,j}(x)=\begin{cases}
			0,&\text{if }x\not\in\sect{v}\setminus\{v\},\\
			a_{|x|-|v|-1}E_{v,j}(p^{|x|-|v|-1}(x)),&\text{otherwise}.
		\end{cases}
	\end{equation}
	Hence $ f_{v,j}$ is harmonic for every $v\in X$ and $j\in I_v$. Furthermore $ f_{v,j}$ is bounded, since for every $x\in X$ 
	\[| f_{v,j}(x)|\leq (1-q^{-1})^{-1}\|e_{v,j}\|_{\infty}.\] 
	Hence $ f_{v,j}\in\mathcal{A}^2(\sigma)$ for every reference measure $\sigma$. 
	Notice that, upon setting $f_0(x)\equiv 1$, the family
	\begin{equation}\label{tildeF}
		\cB=\{ f_0\}\cup\{ f_{v,j}\colon v\in X,\, j \in I_v\}\subseteq\mathcal{A}^2(\sigma)
	\end{equation}
	is independent of the choice of the reference measure $\sigma$. 
	
	\begin{prop}
		The family 
		$ \cB$ is a complete orthogonal system in $\mathcal{A}^2(\sigma)$ for every reference measure $\sigma$.
	\end{prop}
	\begin{proof}
		Fix a reference measure $\sigma$. The fact that $ f_0$ is orthogonal to every other function of the family follows from the harmonicity of $ f_{v,j}$ and~\eqref{harmoball}. Indeed  
		\[	\langle  f_{v,j}, f_0\rangle_{\mathcal{A}^2(\sigma)}=\sum_{x\in X} f_{v,j}(x)\sigma(x)=\sum_{n=0}^{+\infty}\sigma_{ n}\sum_{|x|=n} f_{v,j}(x)=0. \]
		Let us consider $v,\,w\in X$ with $v\neq w$. Without loss of generality we may consider two situations: either $\sect{v}\cap \sect{w}=\emptyset$ or $\sect{v}\subsetneq \sect{w}$. In the first case $ f_{v,j}\perp  f_{w,k}$ for every $j\in I_v$ and $k\in I_w$, because their supports are disjoint. If $\sect{v}\subsetneq \sect{w}$, then we can suppose that $|w|<|v|$. Since $ f_{v,j}|_{\ball{o}{|w|+1}}=0$, from Lemma~\ref{lemma} we have
		\begin{align*}
			\langle  f_{v,j}, f_{w,k}\rangle_{\mathcal{A}^2(\sigma)} 
			&=\sum_{|y|=|w|+1}(b_{|w|} f_{v,j}(y)-b'_{|w|} f_{v,j}(p(y)))\overline{E_{w,k}(y)}=0.	
		\end{align*}
		It remains to prove the orthogonality in the case $v=w$. Let $j,\,k\in I_v$ be such that $j\neq k$. We know that $ f_{v,k}|_{\ball{o}{|v|}}=0$, so that by Lemma~\ref{lemma}
		\begin{align*}
			\langle  f_{v,j}, f_{v,k}\rangle_{\mathcal{A}^2(\sigma)}
			&=b_{|v|}\sum_{|y|=|v|+1}E_{v,j}(y)\overline{E_{v,k}(y)}=b_{|v|}\sum_{y\in s(v)}e_{v,j}(y)\overline{e_{v,k}(y)}=0,
		\end{align*} 
		where we used the fact that $\supp(E_{v,k}),\supp(E_{v,j})\subseteq s(v)$ and the orthogonality of $e_{v,j}$ and $e_{v,k}$ in $W_v$.
		%
		
		We now show that $ \cB$ is complete.
		Take $g\in \mathcal{A}^2(\sigma)$ such that $\langle g,f\rangle_{\mathcal{A}^2(\sigma)}=0$ for every $f\in \cB$. We show that $g$ is the null function in $\mathcal{A}^2(\sigma)$. In particular we prove by induction that $g=0$ on every $\ball{o}{m}$, $m\in\N$.\\
		We start by observing that $\langle g,  f_0\rangle_{\mathcal{A}^2(\sigma)}=0$ implies $g(o)=0$. Indeed by \eqref{harmoball}
		\begin{equation}\label{nucleoino}
			0=\langle g,  f_0\rangle_{\mathcal{A}^2(\sigma)}=\sum_{n=0}^{+\infty} \sigma_{ n}\sum_{|x|=n}g(x)=\left(1+\frac{q+1}{q}\sum_{n=1}^{+\infty}q^{n}\sigma_n\right)g(o)=B_\sigma g(o).
		\end{equation}
		We assume now $g=0$ on $\ball{o}{m}$ for some $m\in\mathbb{N}$. 
		Let $v\in\sph{o}{m}$. Observe that since $g$ is harmonic and $g(v)=0$, we have $g|_{s(v)}\in W_v$. Hence for every $ j\in I_v$
		\begin{align}\label{b_v}
			0&=\langle g,  f_{v,j}\rangle_{\mathcal{A}^2(\sigma)}=b_m\sum_{y\in s(v) }\overline{e_{v,j}(y)}g(y)
		\end{align}
		and this implies that $g(y)=0$ for every $y\in s(v)$ and so for every $y\in \sph{o}{m+1}$, namely $g$ vanishes on $\ball{o}{m+1}$. The fact that $g\equiv 0$ follows by induction.
	\end{proof}
	
	We now fix a measure $\sigma$ and compute the norm of functions of the family $ \cB$ in $\mathcal{A}^2(\sigma)$. Evidently, $\| f_0\|_{\cA^2(\sigma)}^2=B_\sigma$.
	Let $v\in X$ and $j\in I_v$. 
	By \eqref{b_v}, we have 
	
	\begin{equation}\label{squarenorm}
		\| f_{v,j}\|_{\mathcal{A}^2(\sigma)}^2=\langle  f_{v,j}, f_{v,j} \rangle_{\mathcal{A}^2(\sigma)}=b_{|v|}\sum_{y\in s(v)}\overline{e_{v,j}(y)}e_{v,j}(y)=b_{|v|}.
	\end{equation}
	Hence the norm of $ f_{v,j}$ does not depend on $j$ and coincides with the constant in~\eqref{bt}.
	Hence 
	\begin{equation}\label{basis}
		\cB_\sigma=\{B_\sigma^{-\frac12}f_0\}\cup\{b_{|v|}^{-\frac12}f_{v,j}\colon v\in X,\, j\in I_v\}
	\end{equation}
	is an orthonormal basis of $\mathcal{A}^2(\sigma)$.

	\section{The reproducing kernel of $\mathcal{A}^2(\sigma)$}\label{sec:kernel}
	In this section we fix a reference measure $\sigma$. We show that the Bergman space $\mathcal{A}^2(\sigma)$ is a RKHS and we first obtain a recursive formula for the kernel and we then derive a formula in closed form. 
	Observe that the main ingredient used in the proofs are the harmonic extension and the orthonormal basis defined in the previous section together with the fact that $W_v$ are reproducing kernel Hilbert spaces, too. 
	
	Let $z\in X$. We consider the evaluation functional $\Phi_z\colon\mathcal{A}^2(\sigma)\to\mathbb{C}$ defined by $\Phi_zg=g(z)$.
	Observe that $\Phi_z$ is a bounded operator. Indeed by the Cauchy-Schwarz inequality
	\[|\Phi_zg|^2=|g(z)|^2\leq\frac{1}{\sigma(z)}\sum_{x\in X}|g(x)|^2\sigma(x)=\frac{\|g\|_2^2}{\sigma(z)}.\]
	Thus $\mathcal{A}^2(\sigma)$ is a RKHS, that is for every $z\in X$ there exists $K_z\in\mathcal{A}^2(\sigma)$ such that 
	\[\langle g,K_z\rangle_{\mathcal{A}^2(\sigma)}=g(z), \qquad g\in\mathcal{A}^2(\sigma).\]
	Let $K\colon X\times X\to\C$ be the kernel defined by $K(z,x):=K_z(x)$.
	
	Since $\cB_\sigma$ defined in~\eqref{basis} is an orthonormal basis of $\mathcal{A}_\sigma^2$, for every $z\in X$ we can write 
	\begin{equation}\label{decomp}
		K_z=\sum_{f\in\cB_\sigma}\langle K_z,f\rangle_{\cA^2(\sigma)}f=\sum_{f\in\cB_\sigma}\overline{f(z)}f
		=\frac{1}{B_\sigma}+\sum_{v\in X}\sum_{j\in I_v}\frac{ \overline{f_{v,j}(z) }f_{v,j}}{b_{|v|}}.
	\end{equation}
	We recall that by \eqref{fvj}, for every $z\in X$ 
	\[\{v\in X\colon f_{v,j}(z)\neq 0 \text{ for some }j\in I_v\}=\begin{cases*}\emptyset,&\text{if }$v=o$;\\
		[o,p(z)],&\text{if }$v\neq o$.
	\end{cases*}\]
	Hence for every $z\in X$ the sum in \eqref{decomp} is finite and the decomposition of $K_z$ holds true pointwise.
	
	Our goal is to compute $K_z$. 
	To this end, we introduce the auxiliary function $\Gamma\colon X\times X\times X\to\R$ which is a parametrization of the family of reproducing kernels for the spaces $\{W_v\}_{v\in X}$.
	For every $(v,z,x)\in X\times X\times X$ we set 
	\begin{equation*}
		\Gamma(v,z,x)=\begin{dcases}
			0,& \text{if }\{z,x\}\nsubseteq\sect{v}\setminus\{v\};\\
			\frac{\# s(v)-1}{\# s(v)},& \text{if }\{z,x\}\subseteq\sect{y}\text{ for some }y\in s(v);\\
			-\frac{1}{\# s(v)},& \text{otherwise}.
		\end{dcases}
	\end{equation*}
	Observe that $\Gamma$ is symmetric in the second and third variables. Furthermore, $\Gamma(v,z,\cdot)$ is the null function if $z\not\in\sect{v}\setminus\{v\}$ and whenever $z\in\sect{v}\setminus\{v\}$ we have $\supp(\Gamma(v,z,\cdot))=\sect{v}\setminus\{v\}$.  Moreover, the values of $\Gamma(v,z,\cdot)$ on $\sect{v}\setminus\{v\}$ are completely determined by the values on $s(v)$, as the value of $\Gamma(v,z,\cdot)$ at $x\in\sect{v}\setminus\{v\}$ is equal to the value at $p^{|x|-|v|-1}(x)\in s(v)$.
	\begin{center}
		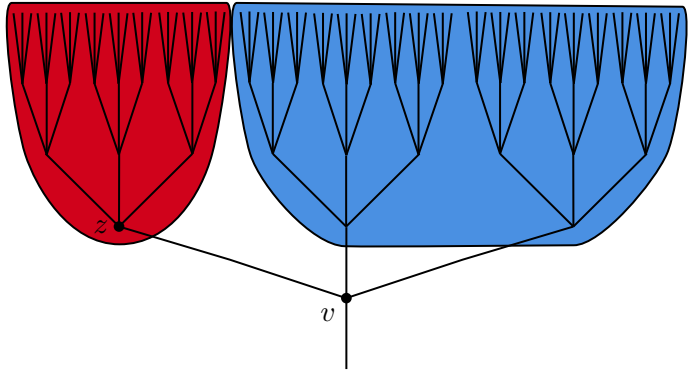
\begin{figure}[h]
			\centering	
			\tikzset{every picture/.style={line width=0.75pt}} 
			
			\begin{tikzpicture}[x=0.75pt,y=0.75pt,scale=1.2,yscale=-1,xscale=1]
				
				\draw  [draw opacity=0][fill={rgb, 255:red, 74; green, 144; blue, 226 }  ,fill opacity=0.2 ] (275,67) .. controls (280.6,67) and (455.4,68.6) .. (460.2,68.6) .. controls (465,68.6) and (459,118.2) .. (453.8,130.2) .. controls (448.6,142.2) and (426.6,168.6) .. (415,168.6) .. controls (403.4,168.6) and (333.4,169.4) .. (319.8,169) .. controls (306.2,168.6) and (285,144.2) .. (280.6,130.6) .. controls (276.2,117) and (269.4,67) .. (275,67) -- cycle ;
				\draw  [draw opacity=0][fill={rgb, 255:red, 208; green, 2; blue, 27 }  ,fill opacity=0.19 ] (181.4,67) .. controls (187,67) and (265.8,67) .. (270.6,67) .. controls (275.4,67) and (269,114.2) .. (264.6,130.6) .. controls (260.2,147) and (247.4,168.2) .. (226.2,168.2) .. controls (205,168.2) and (191.4,144.2) .. (187,130.6) .. controls (182.6,117) and (175.8,67) .. (181.4,67) -- cycle ;
				\draw    (320.5,190.75) -- (320.5,220.5) ;
				\draw    (414.9,160.75) -- (368.8,174.65) -- (320.5,190.75) ;
				\draw    (320.5,161) -- (320.5,190.75) ;
				\draw    (414.83,129) -- (414.9,160.75) ;
				\draw    (445.4,130.25) -- (414.9,160.75) ;
				\draw    (384.4,130.25) -- (414.9,160.75) ;
				\draw    (384.4,100.5) -- (384.4,130.25) ;
				\draw    (374.7,100.35) -- (384.4,130.25) ;
				\draw    (394.6,100.3) -- (384.4,130.25) ;
				\draw    (374.7,70.6) -- (374.7,100.35) ;
				\draw    (371,70.85) -- (374.7,100.35) ;
				\draw    (378.4,70.85) -- (374.7,100.35) ;
				\draw    (384.5,70.8) -- (384.5,100.55) ;
				\draw    (380.8,71.05) -- (384.5,100.55) ;
				\draw    (388.2,71.05) -- (384.5,100.55) ;
				\draw    (394.5,70.8) -- (394.5,100.55) ;
				\draw    (390.8,71.05) -- (394.5,100.55) ;
				\draw    (398.2,71.05) -- (394.5,100.55) ;
				\draw    (414.97,100.79) -- (414.97,130.54) ;
				\draw    (405.27,100.64) -- (414.97,130.54) ;
				\draw    (425.17,100.59) -- (414.97,130.54) ;
				\draw    (405.27,70.89) -- (405.27,100.64) ;
				\draw    (401.57,71.14) -- (405.27,100.64) ;
				\draw    (408.97,71.14) -- (405.27,100.64) ;
				\draw    (415.07,71.09) -- (415.07,100.84) ;
				\draw    (411.37,71.34) -- (415.07,100.84) ;
				\draw    (418.77,71.34) -- (415.07,100.84) ;
				\draw    (425.07,71.09) -- (425.07,100.84) ;
				\draw    (421.37,71.34) -- (425.07,100.84) ;
				\draw    (428.77,71.34) -- (425.07,100.84) ;
				\draw    (444.97,100.79) -- (444.97,130.54) ;
				\draw    (435.27,100.64) -- (444.97,130.54) ;
				\draw    (455.17,100.59) -- (444.97,130.54) ;
				\draw    (435.27,70.89) -- (435.27,100.64) ;
				\draw    (431.57,71.14) -- (435.27,100.64) ;
				\draw    (438.97,71.14) -- (435.27,100.64) ;
				\draw    (445.07,71.09) -- (445.07,100.84) ;
				\draw    (441.37,71.34) -- (445.07,100.84) ;
				\draw    (448.77,71.34) -- (445.07,100.84) ;
				\draw    (455.07,71.09) -- (455.07,100.84) ;
				\draw    (451.37,71.34) -- (455.07,100.84) ;
				\draw    (458.77,71.34) -- (455.07,100.84) ;
				\draw    (320.41,131) -- (320.41,160.75) ;
				\draw    (350.91,130.25) -- (320.41,160.75) ;
				\draw    (289.91,130.25) -- (320.41,160.75) ;
				\draw    (289.91,100.5) -- (289.91,130.25) ;
				\draw    (280.21,100.35) -- (289.91,130.25) ;
				\draw    (300.11,100.3) -- (289.91,130.25) ;
				\draw    (280.21,70.6) -- (280.21,100.35) ;
				\draw    (276.51,70.85) -- (280.21,100.35) ;
				\draw    (283.91,70.85) -- (280.21,100.35) ;
				\draw    (290.01,70.8) -- (290.01,100.55) ;
				\draw    (286.31,71.05) -- (290.01,100.55) ;
				\draw    (293.71,71.05) -- (290.01,100.55) ;
				\draw    (300.01,70.8) -- (300.01,100.55) ;
				\draw    (296.31,71.05) -- (300.01,100.55) ;
				\draw    (303.71,71.05) -- (300.01,100.55) ;
				\draw    (320.49,101.07) -- (320.49,130.82) ;
				\draw    (310.79,100.92) -- (320.49,130.82) ;
				\draw    (330.69,100.87) -- (320.49,130.82) ;
				\draw    (310.79,71.17) -- (310.79,100.92) ;
				\draw    (307.09,71.42) -- (310.79,100.92) ;
				\draw    (314.49,71.42) -- (310.79,100.92) ;
				\draw    (320.59,71.37) -- (320.59,101.12) ;
				\draw    (316.89,71.62) -- (320.59,101.12) ;
				\draw    (324.29,71.62) -- (320.59,101.12) ;
				\draw    (330.59,71.37) -- (330.59,101.12) ;
				\draw    (326.89,71.62) -- (330.59,101.12) ;
				\draw    (334.29,71.62) -- (330.59,101.12) ;
				\draw    (350.49,100.79) -- (350.49,130.54) ;
				\draw    (340.79,100.64) -- (350.49,130.54) ;
				\draw    (360.69,100.59) -- (350.49,130.54) ;
				\draw    (340.79,70.89) -- (340.79,100.64) ;
				\draw    (337.09,71.14) -- (340.79,100.64) ;
				\draw    (344.49,71.14) -- (340.79,100.64) ;
				\draw    (350.59,71.09) -- (350.59,100.84) ;
				\draw    (346.89,71.34) -- (350.59,100.84) ;
				\draw    (354.29,71.34) -- (350.59,100.84) ;
				\draw    (360.59,71.09) -- (360.59,100.84) ;
				\draw    (356.89,71.34) -- (360.59,100.84) ;
				\draw    (364.29,71.34) -- (360.59,100.84) ;
				\draw    (225.97,160.75) -- (272.07,174.65) -- (320.37,190.75) ;
				\draw    (226.04,129) -- (225.97,160.75) ;
				\draw    (195.47,130.25) -- (225.97,160.75) ;
				\draw    (256.47,130.25) -- (225.97,160.75) ;
				\draw    (256.47,100.5) -- (256.47,130.25) ;
				\draw    (266.17,100.35) -- (256.47,130.25) ;
				\draw    (246.27,100.3) -- (256.47,130.25) ;
				\draw    (266.17,70.6) -- (266.17,100.35) ;
				\draw    (269.87,70.85) -- (266.17,100.35) ;
				\draw    (262.47,70.85) -- (266.17,100.35) ;
				\draw    (256.37,70.8) -- (256.37,100.55) ;
				\draw    (260.07,71.05) -- (256.37,100.55) ;
				\draw    (252.67,71.05) -- (256.37,100.55) ;
				\draw    (246.37,70.8) -- (246.37,100.55) ;
				\draw    (250.07,71.05) -- (246.37,100.55) ;
				\draw    (242.67,71.05) -- (246.37,100.55) ;
				\draw    (225.9,100.79) -- (225.9,130.54) ;
				\draw    (235.6,100.64) -- (225.9,130.54) ;
				\draw    (215.7,100.59) -- (225.9,130.54) ;
				\draw    (235.6,70.89) -- (235.6,100.64) ;
				\draw    (239.3,71.14) -- (235.6,100.64) ;
				\draw    (231.9,71.14) -- (235.6,100.64) ;
				\draw    (225.8,71.09) -- (225.8,100.84) ;
				\draw    (229.5,71.34) -- (225.8,100.84) ;
				\draw    (222.1,71.34) -- (225.8,100.84) ;
				\draw    (215.8,71.09) -- (215.8,100.84) ;
				\draw    (219.5,71.34) -- (215.8,100.84) ;
				\draw    (212.1,71.34) -- (215.8,100.84) ;
				\draw    (195.9,100.79) -- (195.9,130.54) ;
				\draw    (205.6,100.64) -- (195.9,130.54) ;
				\draw    (185.7,100.59) -- (195.9,130.54) ;
				\draw    (205.6,70.89) -- (205.6,100.64) ;
				\draw    (209.3,71.14) -- (205.6,100.64) ;
				\draw    (201.9,71.14) -- (205.6,100.64) ;
				\draw    (195.8,71.09) -- (195.8,100.84) ;
				\draw    (199.5,71.34) -- (195.8,100.84) ;
				\draw    (192.1,71.34) -- (195.8,100.84) ;
				\draw    (185.8,71.09) -- (185.8,100.84) ;
				\draw    (189.5,71.34) -- (185.8,100.84) ;
				\draw    (182.1,71.34) -- (185.8,100.84) ;
				
				\node[anchor=east] at (225.97,160.75) {$z$};
				\node[circle,fill=black,inner sep=0pt,minimum size=4pt,label=below:{}] (a) at (225.97,160.75) {};
				\node[anchor=north east] at (320.5,190.75) {$v$};
				\node[circle,fill=black,inner sep=0pt,minimum size=4pt,label=below:{}] (a) at (320.5,190.75) {};
			\end{tikzpicture}
			
			\caption{Partial representation of the function $\Gamma(v,z,\cdot)$ on $\sect{v}$. The value of $\Gamma(v,z,\cdot)$ at the vertices in the red area is $\frac{\# s(v)-1}{\# s(v)}$, while in the blue area is $-\frac{1}{\# s(v)}$. Clearly, $\Gamma(v,z,v)=0$.}
		\end{figure}
	\end{center}
	
	We now show that $\Gamma(v,\cdot,\cdot)$ is the reproducing kernel of $W_v$, namely that for $z\in s(v)$ we have
	\[\varphi(z)=\langle \varphi,\Gamma(v,z,\cdot)\rangle_{W_v},\qquad\varphi\in W_v.\]
	First of all $\Gamma(v,z,\cdot)\in W_v$ because 
	\[\sum_{y\in s(v)}\Gamma(v,z,y)=-(\# s(v)-1)\frac{1}{\# s(v)}+\frac{\# s(v)-1}{\# s(v)}=0.\]
	Furthermore,
	\begin{align*}
		\langle \varphi,\Gamma(v,z,\cdot)\rangle_{W_v}&=\frac{\# s(v)-1}{\# s(v)}\varphi(z)-\frac{1}{\# s(v)}\sum_{\substack{y\in s(v)\\y\neq z}}\varphi(y)\\&= \frac{\# s(v)-1}{\# s(v)}\varphi(z)+\frac{1}{\# s(v)}\varphi(z)=\varphi(z),  
	\end{align*}
	because $\varphi\in W_v$.
	
	It is easy to see that $\Gamma(v,z,\cdot)$ is harmonic on $\ball{o}{|v|}$ so that we can consider the harmonic extension $(\Gamma(v,z,\cdot))^H_{|v|}$, which is bounded by construction. Indeed from the definition of harmonic extension we have for every $x\in \sect{v}\setminus\{v\}$
	\begin{equation}\label{eq:Gammaharm}
		(\Gamma(v,z,\cdot))^H_{|v|}(x)=\left(\sum_{j=0}^{|x|-|v|-1}q^{-j}\right)\Gamma(v,z,p^{|x|-|v|-1}(x))=a_{|x|-|v|-1}\Gamma(v,z,x),
	\end{equation}
	and it vanishes elsewhere. We recall that if $z\not\in\sect{v}$, then $\Gamma(v,z,\cdot)=(\Gamma(v,z,\cdot))^H_{|v|}$ is the null function.
	
	\begin{prop}\label{ricorsiva}
		Let $z\in X$ and $[o,z]=\{v_t\}_{t=0}^{|z|}$. The kernel $K_z$ is
		\begin{equation*}
			K_z=\begin{cases}
				\displaystyle \frac{1}{B_\sigma},&\text{if }z=o,\\
				\displaystyle K_o+\frac{1}{b_0}(\Gamma(o,z,\cdot))^H_{0},&\text{if }|z|=1,\\
				\displaystyle -\frac{1}{q}K_{v_{m-2}}+\frac{q+1}{q}K_{v_{m-1}}+\frac{1}{b_{m-1}}(\Gamma(v_{m-1},z,\cdot))^H_{m-1},&\text{if }|z|=m>1.
			\end{cases}
		\end{equation*}
	\end{prop}
	\begin{proof}
		Since the measure $\sigma$ is finite and the constant functions are harmonic, $K_o=\frac{1}{B_\sigma}\in\mathcal{A}^2(\sigma)$. The reproducing property follows from the computations in~\eqref{nucleoino}.\\
		Now we observe that for every $v,\,z\in X$ such that $z\in\sect{v}$ and $g\in\cA^2(\sigma)$,
		by Lemma~\ref{lemma} and $\supp(\Gamma(v,z,\cdot))=\sect{v}\setminus\{v\}$, we have
		\begin{align}\label{prodscal}
			\langle g,(\Gamma(v,z,\cdot))_{|v|}^H \rangle_{\cA^2(\sigma)}
			&=\sum_{|y|=|v|+1}(b_{|v|}g(y)-b_{|v|}'g(p(y)))\Gamma(v,z,y)\nonumber\\
			&=b_{|v|}\sum_{y\in s(v)}g(y)\Gamma(v,z,y)-b'_{|v|}g(v)\sum_{y\in s(v)}\Gamma(v,z,y)\nonumber\\
			&=b_{|v|}\sum_{y\in s(v)}g(y)\Gamma(v,z,y),
		\end{align}
		where we used $\Gamma(v,z,\cdot)|_{s(v)}\in W_v$.\\
		We now consider the case when $|z|=1$. The function $K_z\in\mathcal{A}^2(\sigma)$ because it is sum of functions in $\mathcal{A}^2(\sigma)$.
		We prove the reproducing property. For $g\in\mathcal{A}^2(\sigma)$, by the reproducing formula of $K_o$ and~\eqref{prodscal} with $v=o$,
		\begin{align*}
			\langle g,K_z \rangle_{\mathcal{A}^2(\sigma)}&=g(o)+ \frac{1}{b_0}\langle g,(\Gamma(o,z,\cdot))^H_{0}\rangle_{\mathcal{A}^2(\sigma)}\\
			&=g(o)+
			\sum_{|y|=1}g(y)\Gamma(o,z,y)\\
			&=g(o)+\frac{q}{q+1}g(z)-\frac{1}{q+1}\sum_{\substack{|y|=1\\y\neq z}}g(y)=g(z),
		\end{align*}
		where we used that $g$ is harmonic at $o$. 
		\\
		It remains to consider the case when $|z|=m>1$. We have $K_z\in\mathcal{A}^2(\sigma)$ since it is the sum of bounded and harmonic functions. For $g\in\mathcal{A}^2(\sigma)$ by induction on $m$ and~\eqref{prodscal} with $v=v_{m-1}$ we have 
		\begin{align*}
			\langle g,K_z \rangle_{\mathcal{A}^2(\sigma)}&=-\frac{1}{q}g(v_{m-2})+\frac{q+1}{q}g(v_{m-1})+ \frac{1}{b_{m-1}}\langle g,(\Gamma(v_{m-1},z,\cdot))^H_{m-1} \rangle_{\mathcal{A}^2(\sigma)}\\
			&=-\frac{1}{q}g(v_{m-2})+\frac{q+1}{q}g(v_{m-1})+ \sum_{y\in s(v_{m-1})}\Gamma(v_{m-1},z,y)g(y)\\
			&=-\frac{1}{q}g(v_{m-2})+\frac{1}{q}\sum_{y\sim v_{m-1}}g(y)+\frac{q-1}{q}g(z)-\frac{1}{q}\sum_{\substack{y\in s(v_{m-1})\\y\neq z}}g(y)=g(z),
		\end{align*}
		where we used the fact that $g$ is harmonic at $v_{m-1}$.
	\end{proof}
	
	In Proposition~\ref{ricorsiva} the kernel $K_z$ is expressed through a two-step recursive formula. We want to find an explicit formula for $K_z$. 
	
	%
	\begin{thm}\label{kernelformula}
		For every $(z,x)\in X\times X$
		\begin{equation}\label{kernel}
			K(z,x)=\frac{1}{B_\sigma}+\frac{q^2}{(q-1)^2}\sum_{v\in X}\frac{1}{b_{|v|}}\Gamma(v,z,x)(1-q^{|v|-|z|})(1-q^{|v|-|x|}). 
		\end{equation}
	\end{thm}
	\begin{proof}
		Let $z\in X$ and $[o,z]=\{v_t\}_{t=0}^{|z|}$. We start by proving that
		\begin{align}\label{kernelesplicito} K_z(x)&=\frac{1}{B_\sigma}+\sum_{t=0}^{|z|-1}a_{|z|-t-1}\frac{1}{b_t}(\Gamma(v_{t},v_{t+1},x))^H_{t},\qquad x\in X.
		\end{align}
		The case $z=o$ follows trivially from Proposition~\ref{ricorsiva} and the convention on sums stated in the Introduction.
		We prove \eqref{kernelesplicito} by induction on $m=|z|\geq 1$. The case $m=1$ directly follows from Proposition~\ref{ricorsiva}, too. Let $m\in\mathbb{N}$, $m>1$ and $z\in X$, with $|z|=m$, and suppose that \eqref{kernelesplicito} holds for every vertex in $\ball{o}{m-1}$. Hence by Proposition~\ref{ricorsiva} we have
		\begin{align*}
			K_z=&-\frac{1}{q}K_{v_{m-2}}+\frac{q+1}{q}K_{v_{m-1}}+\frac{1}{b_{m-1}}(\Gamma(v_{m-1},z,\cdot))^H_{m-1}\\
			=&-\frac{1}{q}\left[\frac{1}{B_\sigma}+\sum_{t=0}^{m-3}a_{m-t-3}\frac{1}{b_t}(\Gamma(v_{t},v_{t+1},\cdot))^H_{t}\right]\\
			&+\frac{q+1}{q}\left[\frac{1}{B_\sigma}+\sum_{t=0}^{m-2}a_{m-t-2}\frac{1}{b_t}(\Gamma(v_{t},v_{t+1},\cdot))^H_{t}\right]+\frac{1}{b_{m-1}}(\Gamma(v_{m-1},z,\cdot))^H_{m-1}\\
			=&\frac{1}{B_\sigma}+\sum_{t=0}^{m-2}\left(\frac{q+1}{q}a_{m-t-2}-\frac{1}{q}a_{m-t-3}\right)\frac{1}{b_t}(\Gamma(v_{t},v_{t+1},\cdot))^H_{t}\\&+\frac{1}{b_{m-1}}(\Gamma(v_{m-1},z,\cdot))^H_{m-1}\\
			=&\frac{1}{B_\sigma}+\sum_{t=0}^{m-1}a_{m-t-1}\frac{1}{b_t}(\Gamma(v_{t},v_{t+1},\cdot))^H_{t},
		\end{align*}
		where we used $(q+1)a_{n-1}-a_{n-2}=qa_{n}$.
		Hence we proved~\eqref{kernelesplicito} by induction.
		Since ${\rm supp}((\Gamma(v_{t},v_{t+1},\cdot))^H_{t})=\sect{v_t}\setminus\{v_t\}$,
		we have that the $t$-th term of the sum in~\eqref{kernelesplicito} does not vanish if and only if $x\in T_{v_t}\setminus\{v_t\}$, that is when $v_t\in[o,x]$, and hence by~\eqref{eq:Gammaharm}, we have
		\begin{align*}
			K(z,x)&=\frac{1}{B_\sigma}+\sum_{v\in X} \frac{1}{b_{|v|}}a_{|z|-|v|-1}a_{|x|-|v|-1}\Gamma(v,z,x)\\
			&=\frac{1}{B_\sigma}+\frac{q^2}{(q-1)^2}\sum_{v\in X}\frac{1}{b_{|v|}}(1-q^{|v|-|z|})(1-q^{|v|-|x|})\Gamma(v,z,x).
		\end{align*}
	\end{proof}
	
	\begin{oss}
		The confluent of two vertices $z,x\in X$ is the common vertex of $[o,x]$ and $[o,z]$ farthest from $o$, denoted by $z\wedge x$. 
		It is possible to see that the value of the kernel $K$ at $(z,x)\in X\times X$ depends only on the values of $|x|$, $|z|$ and $|z\wedge x|$. Indeed, from~\eqref{kernel} and the fact that $\Gamma(v,z,x)$ does not vanish if and only if $v\in[o,z]\cap[o,x]=[o,z\wedge x]$, we have
		\begin{equation}\label{kernelconfluente}
			\begin{split}
			K_z(x)&=\frac{1}{B_\sigma}+\frac{q^2}{(q-1)^2}\sum_{t=0}^{|z\wedge x|}\frac{1}{b_t}a_{|z|-t-1}a_{|x|-t-1}(1-q^{t-|z|})(1-q^{t-|x|})\\
			&=\frac{1}{B_\sigma}+\frac{q^2}{(q-1)^2}\sum_{t=0}^{|z\wedge x|}\frac{1}{b_t}\Gamma(v_t,z,x)(1-q^{t-|z|})(1-q^{t-|x|}),
		\end{split}
		\end{equation}
		where $\{v_t\}_{t=0}^{|z|}=[o,z]$.
		Furthermore, it is clear that $K$ is symmetric, that is $K(z,x)=K(x,z)$.
	\end{oss}

	In the following sections we restrict our attention to the family of the exponentially decreasing radial measures $\mu_\alpha$, $\alpha>1$, defined in Example~\ref{ex}.
	
	We shall use the notation $L_\alpha^p$ and $\mathcal{A}_\alpha^p$ for the Lebesgue and Bergman spaces with respect to $\mu_\alpha$, respectively. Furthermore, we denote by $K_\alpha\colon X\times X\to\R$ the reproducing kernel of $\mathcal{A}_\alpha^2$. It will be useful to keep track of the weight in the constants introduced in~\eqref{bt}, so we denote them by $b_{\alpha,n}$. In particular observe that in this case there is a relation between the constants: for every $n\in\mathbb{N}$
	\begin{equation}\label{Bt}\begin{split}
			b_{\alpha,n}&=\sum_{m=n+1}^{+\infty}\left[q^{-\alpha m}\left(\sum_{k=0}^{m-n-1}q^k\right)\left(\sum_{j=0}^{m-n-1}q^{-j}\right) \right]\\
			&=\sum_{\ell=1}^{+\infty}\left[q^{-\alpha (\ell+n)}\left(\sum_{k=0}^{\ell-1}q^k\right)\left(\sum_{j=0}^{\ell-1}q^{-j}\right) \right]=q^{-\alpha n}b_{\alpha,0}.
		\end{split}
	\end{equation}
	Furthermore we set $B_\alpha=\mu_{\alpha}(X)$.
	
	Now we show that the kernel $K_\alpha$ satisfies an integral condition which will be formalized in Section~\ref{seccz}, see~\eqref{hormander}.
	\begin{prop}\label{prophormalpha}
		The following holds
		\begin{equation}\label{hormalpha}
			\sup_{v\in X\setminus\{o\}}\sup_{x,y\in \sect{v}}\sum_{z\in X\setminus \sect{v}}|K_\alpha(z,x)-K_\alpha(z,y)|q^{-\alpha|z|}<+\infty. 
		\end{equation}
	\end{prop}
	\begin{proof}
		Let $v\in X\setminus\{o\}$. We start by proving that~\eqref{hormalpha} holds for $y=v$.
		Consider $x\in \sect{v}$ and observe that if $z\in X\setminus \sect{v}$, then $z\wedge x=z\wedge v$ and $\Gamma(u,z,x)=\Gamma(u,z,v)$ for every $u\in [o,z\wedge v]$. Hence, if we put $[o,z\wedge v]=\{u_t\}_{t=0}^{|z\wedge v|}$, then from~\eqref{kernelconfluente} we have
		\begin{align*}
			K_\alpha(z,x)-K_\alpha(z,v)&=\frac{q^2b_{\alpha,0}}{(q-1)^2}\sum_{t=0}^{|z\wedge v|}q^{\alpha t}\Gamma(u_t,z,v)(1-q^{t-|z|})(q^{t-|v|}-q^{t-|x|})\\
			&=\frac{q^2b_{\alpha,0}}{(q-1)^2}\sum_{t=0}^{|z\wedge v|}q^{(1+\alpha) t}\Gamma(u_t,z,v)(1-q^{t-|z|})(q^{-|v|}-q^{-|x|}).
		\end{align*}
		Then, since $|\Gamma(u_t,z,v)|<1$ and $0<q^{-|v|}-q^{-|x|}\leq q^{-|v|}$, we have
		\[\sum_{z\in X\setminus \sect{v}}|K_\alpha(z,x)-K_\alpha(z,v)|q^{-\alpha|z|}\lesssim\sum_{z\in X\setminus \sect{v}}\sum_{t=0}^{|z\wedge v|}q^{(1+\alpha) t}q^{-|v|}q^{-\alpha|z|}.\]
		Observe that, since $z\in X\setminus \sect{v}$, then $|z\wedge v|\in\{0,\dots,|v|-1\}$. Then, for every $\ell \in\{0,\dots,|v|-1\}$, we put 
		\[Y_\ell=\{z\in X\colon |z\wedge v|=\ell\}=\sect{u_\ell}\setminus\sect{u_{\ell+1}}.\]
		Since $\mu_\alpha$ is radial, 
		\[\mu_\alpha(Y_\ell)=\mu_\alpha(u_\ell)+(\# s(u_\ell)-1)\mu_\alpha(\sect{u_{\ell+1}})=\begin{dcases*}
			\frac{1}{1-q^{1-\alpha}},&if $\ell=0$;\\
			q^{-\alpha\ell}\frac{1-q^{-\alpha}}{1-q^{1-\alpha}},&if $1\leq\ell<|v|$.
		\end{dcases*}\]
		Hence we have that 
		\begin{align*}
			\sum_{z\in X\setminus \sect{v}}|K_\alpha(z,x)-K_\alpha(z,v)|q^{-\alpha|z|}&\lesssim
			\sum_{\ell=0}^{|v|-1}q^{(1+\alpha)\ell}q^{-|v|}
			\sum_{z\in Y_\ell}q^{-\alpha|z|}\\
			&\lesssim
			\sum_{\ell=0}^{|v|-1}q^{(1+\alpha)\ell}q^{-|v|}q^{-\alpha \ell}
			\\
			&\simeq 
			\sum_{\ell=0}^{|v|-1}q^{\ell-|v|}\leq \frac{q}{q-1}.
		\end{align*}
		Finally, by the triangular inequality we have~\eqref{hormalpha} for every $y\in \sect{v}$.
	\end{proof}

	\section{Boundedness of the Bergman projector on $L^p_\alpha$}\label{sec:bound}
	
	In this section we study the boundedness properties of the extension of the Bergman projector to $L^p_\alpha$ spaces.
	For the class of exponentially decreasing radial measures we are able to prove that the extension of the Bergman projector to the relative weighted $L^p$-space is bounded if and only if $p>1$ (see Theorem~\ref{theobound}).

	In analogy with the operators studied by Zhu in Section~3.4 of~\cite{zhuoperator}, we introduce two families of operators. For any real parameters $a,b$ and for $c>1$, we define the integral operators
	\[S_{a,b,c}f(z)=q^{-a|z|}\sum_{x\in X}|K_c(z,x)|f(x)q^{-b|x|}, \]
	\[T_{a,b,c}f(z)=q^{-a|z|}\sum_{x\in X}K_c(z,x)f(x)q^{-b|x|}. \]
	
	We prove two results that imply the boundedness properties of the Bergman projectors. Theorem~\ref{3.11} is devoted to the study of the boundedness of $S_{a,b,c}$ and $T_{a,b,c}$ on weighted $L^p$-spaces for $p>1$; the case $p=1$ needs different arguments and for this reason is treated apart in Theorem~\ref{3.12}. The two theorems are the analogues of Theorem~3.11 and Theorem~3.12 in~\cite{zhuoperator}, respectively. The proofs of both theorems are postponed to Subsection~\ref{secproofs}.
	\begin{thm}\label{3.11}
		Let $\alpha>1$, $c>1$ and $1<p<\infty$. The following conditions are equivalent:
		\begin{enumerate}[label=(\roman*)]
			\item\label{3.11 a} the operator $S_{a,b,c}$ is bounded on $L^p_\alpha$;
			\item\label{3.11 b} the operator $T_{a,b,c}$ is bounded on $L^p_\alpha$;
			\item\label{3.11 c} the parameters satisfy
			\[c\leq a+b,\qquad -pa<\alpha-1<p\left(b-1\right). \]
		\end{enumerate}
	\end{thm}
	\begin{thm}\label{3.12} Let $\alpha>1$ and $c>1$. The following conditions are equivalent:
		\begin{enumerate}[label=(\roman*)]
			\item\label{3.12 a} the operator $S_{a,b,c}$ is bounded on $L^1_\alpha$;
			\item\label{3.12 b} the operator $T_{a,b,c}$ is bounded on $L^1_\alpha$;
			\item\label{3.12 c} the parameters either satisfy
			\[c=a+b,\qquad -a<\alpha-1<b-1,\]
			or satisfy
			\[c<a+b,\qquad -a<\alpha-1\leq b-1. \]
		\end{enumerate}
	\end{thm}
	We state a corollary which is simply a reformulation of the previous theorems when $c=a+b$. 
	
	\begin{cor}\label{cor3.13}
		Let $ 1\leq p<\infty$ and $\alpha>1$. If $a,b\in\R$ are such that $a+b>1$, then the following conditions are equivalent:
		\begin{enumerate}[label=(\roman*)]
			\item the operator $S_{a,b,a+b}$ is bounded on $L^p_\alpha$;
			\item the operator $T_{a,b,a+b}$ is bounded on $L^p_\alpha$;
			\item the parameters satisfy
			\[-pa<\alpha-1<p\left(b-1\right).\]
		\end{enumerate}
	\end{cor}
	
	Let $\beta>1$. Since $\mathcal{A}_\beta^2$ is a closed subspace of $L^2_\beta$, there exists an orthogonal projection $P_\beta\colon L^2_\beta\to \mathcal{A}_\beta^2$. Observe that by the reproducing property of $K_{\beta,z}=K_\beta(z,\cdot)$, $z\in X$, we can write the projection $P_\beta f$ of $f\in L_\beta^2$ as follows
	\begin{align*}
		P_\beta f(z)&=\langle P_\beta f,K_{\beta,z}\rangle_{\mathcal{A}_\beta^2}=\langle  f,P_\beta K_{\beta,z}\rangle_{L_\beta^2}=\langle f,K_{\beta,z}\rangle_{L_\beta^2},
	\end{align*}
	where we used the orthogonality of $P_\beta$. Hence we can rewrite $P_\beta$ as the integral operator on $L^2_\beta$ associated to the reproducing kernel $K_\beta$, that is
	\begin{equation}\label{proj}
		P_\beta f(z)=\sum_{x\in X}K_\beta(z,x)f(x)q^{-\beta |x|},\qquad f\in L^2_\beta,\,z\in X.
	\end{equation}
	Since $\mu_\beta$ is finite, $L^p_\beta\subseteq L^2_\beta$ whenever $p\geq 2$. It is then natural to investigate whether the restriction of $P_\beta$ to $L^p_\beta$ is bounded. Furthermore, when $1\leq p<2$ one has $L^2_\beta\subsetneq L^p_\beta$ and we shall study whether $P_\beta$ admits a bounded extension to $L^p_\beta$. A more general question that we want to answer is whether the integral operator $\cK_\beta^\alpha$, $\alpha>1$, with kernel $K_\beta(z,x)q^{(\alpha-\beta)|x|}$
	with respect to the measure $\mu_\alpha$, that is
	\begin{equation}\label{projalphabeta}
		\cK_\beta^\alpha f(z)=\sum_{x\in X}K_\beta(z,x)f(x)q^{(\alpha-\beta) |x|}q^{-\alpha |x|},\qquad f\in L_\alpha^p\cap L^2_\beta,\,z\in X,
	\end{equation}
	extends to a bounded operator from $L_\alpha^p$ to $\cA_\alpha^p$.
	The following result answers the above questions.

	\begin{thm}\label{theobound}
		Let $1\leq p<\infty$, $\alpha,\beta>1$. The operator $\cK_\beta^\alpha$
		extends to a bounded operator from $L^p_\alpha$ to $\mathcal{A}_\alpha^p$ if and only if
		\begin{equation}\label{thmcond}
			p\left(\beta-1\right)>\alpha-1.
		\end{equation}
		In particular, $P_\alpha$ is bounded on $L_\alpha^p$ if and only if $p>1$.
	\end{thm}
	\begin{proof}
		It is sufficient to observe that from~\eqref{projalphabeta}, $\cK_\beta^\alpha=T_{0,\beta,\beta}$ on $L_\alpha^p\cap L_\beta^2$. Hence, the result follows from Corollary~\ref{cor3.13}.
	\end{proof}

	\begin{oss}
		It is worthwhile observing that the unboundedness of $P_\alpha$ on $L_\alpha^1$ may be seen directly with the following example. We make use of Lemma~\ref{lem:nonboundp=1} that will be proved in the next subsection.\\
		For every $n\in\mathbb{N}$, we fix a vertex $v_n$ in $\sph{o}{n}$, and define
		\[f_n(x)=\mathbbm{1}_{\{v_n\}}(x)q^{\alpha|x|},\qquad x\in X.\]
		Clearly, $\|f_n\|_{L_\alpha^1}=1$ and $f_n\in L^2_\alpha$. Hence, $P_\alpha f_n(z)=K_\alpha(z,v_n)$ and  by Lemma~\ref{lem:nonboundp=1}
		\[\|P_\alpha f_n\|_{L^1_\alpha}= \sum_{z\in X}|K_\alpha(z,v_n)|q^{-\alpha|z|}\gtrsim |v_n|=n.\]
		This shows that $P_\alpha$ does not admit a bounded extension to $L_\alpha^1$.
	\end{oss}

	As a direct application of Theorem~\ref{theobound}, we deduce the following result on the dual of Bergman spaces.
	\begin{cor}
		Let $1<p<\infty$ and $\alpha>1$. Then 
		\[(\mathcal{A}_\alpha^p)^*= \mathcal{A}_\alpha^{p'},\]
		where $1<p'<\infty$ is such that $\frac 1p+\frac 1{p'}=1$, with equivalent norms under the pairing 
		\begin{equation}\label{pairing}
			\langle f,g \rangle_{\mathcal{A}_\alpha^p\times\mathcal{A}_\alpha^{p'}}=\sum_{z\in X}f(z)g(z)q^{-\alpha|z|}\qquad f\in \mathcal{A}_\alpha^p, \,g\in \mathcal{A}_\alpha^{p'}.
		\end{equation}
	\end{cor}
	\begin{proof}
		Let $g\in\cA^{p'}_\alpha$. By H\"older inequality we have that \[|\langle f,g \rangle_{\mathcal{A}_\alpha^p\times\mathcal{A}_\alpha^{p'}}|\leq\|g\|_{\mathcal{A}_\alpha^{p'}}\|f\|_{\mathcal{A}_\alpha^p},\] 
		for every $f\in \mathcal{A}_\alpha^p$ so that $g$ defines a functional in $(\mathcal{A}_\alpha^p)^*$.
		Conversely, for $\Phi\in(\mathcal{A}_\alpha^p)^*$, by the Hahn-Banach theorem, there exists $\tilde{\Phi}\in(L_\alpha^p)^*$ such that $\tilde{\Phi}|_{\cA_\alpha^p}=\Phi$ and $\|\Phi\|_{(\mathcal{A}_\alpha^p)^*}\geq\|\tilde{\Phi}\|_{(L_\alpha^p)^*}$. Then, there exists $h\in L^{p'}_\alpha$ such that 
		\[\Phi(f)=\tilde{\Phi}(f)= \langle f,h\rangle_{L_\alpha^p\times L_\alpha^{p'}},\] 
		for every $f\in \mathcal{A}_\alpha^p$. By the orthogonality of $P_\alpha$ and Theorem~\ref{theobound},
		\[\Phi(f)= \langle P_\alpha f,P_\alpha h\rangle_{\mathcal{A}_\alpha^p\times\mathcal{A}_\alpha^{p'}}=\langle f,P_\alpha h\rangle_{\mathcal{A}_\alpha^p\times\mathcal{A}_\alpha^{p'}}.\]
		Hence $\Phi$ corresponds to $P_\alpha h\in\cA_\alpha^2$ under the pairing~\eqref{pairing}.
	\end{proof}
	

	\subsection{Proofs of Theorems~\ref{3.11} and \ref{3.12}}\label{secproofs}

	This subsection is devoted to the proofs of Theorems~\ref{3.11} and \ref{3.12},  splitting them up in various steps. In both statements it is obvious that \ref{3.11 a} implies \ref{3.11 b}. For the rest of the section $\alpha,a,b,c$ denote real parameters with $c>1$.
	\subsubsection{Proof that \ref{3.11 b} implies \ref{3.11 c}}
	In this subsection we suppose that the operator $T_{a,b,c}$ is bounded on $L^p_\alpha$ and we deduce necessary conditions on the parameters $a,b,c, \alpha$. 
	
	\begin{prop}\label{-pa<}
		Let $1\leq p<\infty$. If $T_{a,b,c}f\in L^p_\alpha$ for every $f\in L^p_\alpha$, then $-pa<\alpha-1$.
	\end{prop}
	\begin{proof}
		Consider, for $x\in X$, $f(x)=q^{-R|x|}$ with $R\in\R$ such that 
		\[R>\max\left\{\frac{1-\alpha}{p},1-b\right\}. \]
		Since $Rp>1-\alpha$ we have that $f\in L^p_\alpha$ and for every $z\in X$  
		\begin{align*}
			T_{a,b,c}f(z)&=q^{-a|z|}\sum_{x\in X}K_c(z,x)q^{-(b+R)|x|}\\
			&=q^{-a|z|}\sum_{n=0}^{+\infty}q^{-(b+R)n}\sum_{|x|=n}K_c(z,x)\\
			&=q^{-a|z|}\sum_{n=0}^{+\infty}q^{-(b+R)n}\#\sph{o}{n}K_c(z,o)
		\end{align*}
		by~\eqref{harmoball} applied to the harmonic function $K_c(z,\,\cdot\,)$. Hence, since $R>1-b$,
		\[ T_{a,b,c}f(z)=q^{-a|z|}\frac{1}{B_{c}}\left[1+\frac{q+1}{q}\sum_{n=1}^{+\infty}q^{(-b-R+1)n}\right]=\frac{B_{b+R}}{B_c}q^{-a|z|},\qquad z\in X.\]
		Now observe that $T_{a,b,c}f\in L^p_\alpha$ implies 
		\[\sum_{z\in X}q^{-(ap+\alpha)|z|}=1+\frac{q+1}{q}\sum_{n=1}^{+\infty}q^{(1-ap-\alpha)n}<+\infty,\]
		which holds if and only if $-pa<\alpha-1$, as required.
	\end{proof}
	From now on we write, for  $1\leq p<\infty$
	\[\|e_{v,j}\|_p=\left(  \sum_{y\in s(v)}|e_{v,j}(y)|^p \right)^{1/p},\qquad v\in X,\,j\in I_v.\]
	\begin{prop}\label{a+b>c}
		Let $ 1\leq p<\infty$. If $T_{a,b,c}$ is bounded on $L^p_\alpha$, then $a+b\geq c$.
	\end{prop}
	\begin{proof}
		Fix $R\in\R$ such that  \[R>\max\left\{\frac{1-\alpha}{p},c-b \right\}.\]
		For every $v\in X\setminus\{o\}$ and $j\in I_v$, we define
		$g_{v,j}(x)= f_{v,j}(x)q^{-R|x|}$, $x\in X$, where $ f_{v,j}\in \cB$ are defined in \eqref{fvj}. Since $R>\frac{1-\alpha}{p}$, we have that $g_{v,j}\in L^p_\alpha$. Thus,
		\begin{align*}
			T_{a,b,c}g_{v,j}(z)&=q^{-a|z|}\sum_{x\in X}K_c(z,x) f_{v,j}(x)q^{-(b+R)|x|}\\
			&=q^{-a|z|}\langle  f_{v,j},K_{c,z}\rangle_{L_{b+R}^2}
		\end{align*}
		since $K_{c,z}\in L^2_c\subseteq L^2_{b+R}$ because $R>c-b$. Now we use the decomposition~\eqref{decomp} of $K_{c,z}$ on the orthonormal basis of $\mathcal{A}_c^2$ and obtain
		\begin{align*}
			\langle K_{c,z}, f_{v,j}\rangle_{L_{b+R}^2}&=\langle \frac{1}{B_{c}}+\sum_{u\in X}\sum_{k\in I_u}\frac{ \overline{f_{u,k}(z)} f_{u,k}}{b_{c,|u|}}, f_{v,j}\rangle_{L_{b+R}^2}
			\\
			&=\frac{\overline{ f_{v,j}(z)}}{b_{c,|v|}}\langle  f_{v,j}, f_{v,j}\rangle_{L_{b+R}^2}\\
			&=\frac{b_{b+R,|v|}}{b_{c,|v|}} \overline{ f_{v,j}(z)},
		\end{align*}
		where we use the orthogonality of $ \cB$ and~\eqref{squarenorm}. The norm of $T_{a,b,c}g_{v,j}$ in $L_\alpha^p$ is
		\begin{align*}
			\|T_{a,b,c}g_{v,j}\|_{L_\alpha^p}^p&=\left(\frac{b_{b+R,|v|}}{b_{c,|v|}}\right)^p\sum_{z\in X}| f_{v,j}(z)|^pq^{-(ap+\alpha)|z|}\\
			&=\left(\frac{b_{b+R,|v|}}{b_{c,|v|}}\right)^p\sum_{n=0}^{+\infty}q^{-(ap+\alpha)n}\sum_{|z|=n}| f_{v,j}(z)|^p.
		\end{align*}
		Since ${\rm supp}( f_{v,j})\subseteq\sect{v}\setminus\{v\}$, the sum of $| f_{v,j}|^p$ on the sphere $\sph{o}{n}$ vanishes for every $n\leq |v|$. 
		If $n>|v|$, then the sum is on $\sph{o}{n}\cap\sect{v}$ and if $z\in \sect{v}$ is such that $|z|=n$ then $p^{|z|-|v|-1}(z)$ is the unique vertex in $s(v)$ such that $z$ lies in its sector. Hence by~\eqref{fvj} we have
		\begin{align*}
			\sum_{|z|=n}| f_{v,j}(z)|^p&=\sum_{\substack{|z|=n\\z\in \sect{v}}}|e_{v,j}(p^{|z|-|v|-1}(z))|^pa_{n-|p^{|z|-|v|-1}(z)|}^p\\
			&=a_{n-|v|-1}^p\sum_{\substack{|z|=n\\z\in \sect{v}}}|e_{v,j}(p^{|z|-|v|-1}(z))|^p\\
			&=a_{n-|v|-1}^pq^{n-|v|-1}\sum_{y\in s(v)}|e_{v,j}(y)|^p\\
			&=a_{n-|v|-1}^pq^{n-|v|-1}\|e_{v,j}\|^p_p.
		\end{align*}
		For simplicity, for every $s\in\R$ and $1\leq p<\infty$, we put
		\[C(s,p):=\sum_{m=1}^{+\infty}q^{(1-s)m-1}a_{m-1}^p,\]
		which is finite whenever $s>1$.
		The above computation yields
		\begin{align*}
			\|T_{a,b,c}g_{v,j}\|_{L_\alpha^p}^p&=\left(\frac{b_{b+R,|v|}}{b_{c,|v|}}\right)^p\sum_{n=|v|+1}^{+\infty}q^{-(ap+\alpha)n}a_{n-|v|-1}^pq^{n-|v|-1}\|e_{v,j}\|^p_p\\
			&=\|e_{v,j}\|^p_p\left(\frac{b_{b+R,|v|}}{b_{c,|v|}}\right)^p\sum_{m=1}^{+\infty}q^{-(ap+\alpha)(m+|v|)}a_{m-1}^pq^{m-1}\\
			&=\|e_{v,j}\|_{p}^p\left(\frac{b_{b+R,|v|}}{b_{c,|v|}}\right)^pq^{-(ap+\alpha)|v|}\sum_{m=1}^{+\infty}q^{(1-(ap+\alpha))m-1}a_{m-1}^p\\
			&=\|e_{v,j}\|_{p}^p\,C(ap+\alpha,p)\left(\frac{b_{b+R,|v|}}{b_{c,|v|}}\right)^pq^{-(ap+\alpha)|v|},
		\end{align*}
		where $C(ap+\alpha,p)$ converges because $ap+\alpha>1$, by Proposition~\ref{-pa<}. Furthermore,
		\begin{align*}
			\|g_{v,j}\|_{L_\alpha^p}^p&=\sum_{x\in X}| f_{v,j}(x)|^pq^{-(Rp+\alpha)|x|}\\
			&=\sum_{n=0}^{+\infty}q^{-(Rp+\alpha)n}\sum_{|x|=n}| f_{v,j}(x)|^p\\
			&=\sum_{n=|v|+1}^{+\infty}q^{-(Rp+\alpha)n}a_{n-|v|-1}^pq^{n-|v|-1}\|e_{v,j}\|_{p}^p\\
			&=\|e_{v,j}\|_{p}^pq^{-(Rp+\alpha)|v|}\sum_{m=1}^{+\infty}q^{(1-(Rp+\alpha))m-1}a_{m-1}^p\\
			&=\|e_{v,j}\|_{p}^p\,C(Rp+\alpha,p)q^{-(Rp+\alpha)|v|},
		\end{align*}
		where $C(Rp+\alpha,p)\to 1$ when $R\to +\infty$. From the boundedness of $T_{a,b,c}$ and by~\eqref{Bt}, it follows that for every $v\in X\setminus\{o\}$:
		\begin{align*}
			\frac{\|T_{a,b,c}g_{v,j}\|_{L_\alpha^p}^p}{\|g_{v,j}\|_{L_\alpha^p}^p}&\simeq \left(\frac{b_{b+R,|v|}}{b_{c,|v|}}\right)^pq^{-(ap+\alpha-Rp-\alpha)|v|}\\
			&\simeq q^{-p(R+b-c)|v|}q^{-(ap-Rp)|v|}\\
			&=q^{(c-a-b)p|v|},
		\end{align*}
		which is bounded if and only if $c\leq a+b$.
	\end{proof}

	\begin{prop}\label{lemp>1}
		Let $1<p<\infty$. If $T_{a,b,c}$ is bounded on $L^p_\alpha$, then $\alpha-1<p(b-1)$.
	\end{prop}
	\begin{proof}
		The boundedness of $T_{a,b,c}$ on $L^p_\alpha$ is equivalent to the boundedness of the adjoint operator $T_{a,b,c}^*$ on $L^{p'}_\alpha$. It is easy to see that  
		\[T_{a,b,c}^* g(x)=q^{-(b-\alpha)|x|}\sum_{z\in X}K_c(x,z)g(z)q^{-(a+\alpha)|z|}=T_{b-\alpha,a+\alpha,c}g(x),\qquad g\in L^{p'}_\alpha.\]
		Hence, the fact that $T_{a,b,c}^*$ is bounded on $L^{p'}_\alpha$ implies, by Proposition~\ref{-pa<}, that $-p'(b-\alpha)<\alpha-1$, that is $\alpha-1<p(b-1)$. \end{proof}
	Propositions \ref{-pa<}, \ref{a+b>c}, \ref{lemp>1} show that~\ref{3.11 b} implies~\ref{3.11 c} in Theorem~\ref{3.11}.
	Now we focus on the case $p=1$, and we prove that (ii) implies (iii) in Theorem~\ref{3.12}.
	
	\begin{lem}\label{lem:nonboundp=1} Let $\alpha>1$. Then:
		\[\sum_{z\in X}|K_\alpha(x,z)|q^{-\alpha|z|}\gtrsim |x|,\qquad x\in X.\]
	\end{lem}
	\begin{proof}
		The case $x=o$ is trivial.
		For every $x\in X\setminus\{o\}$, we put $\{v_t\}_{t=0}^{|x|}=[o,x]$. Then, by \eqref{kernelconfluente}
		\begin{align*}
			\sum_{z\in X}&|K_\alpha (x,z)|q^{-\alpha |z|}\geq \sum_{t=1}^{|x|}|K_\alpha (x,v_t)|q^{-\alpha t}\\
			&= \sum_{t=1}^{|x|}\left(\frac{1}{B_\alpha}+\frac{q^2}{(q+1)^2}\sum_{v\in X}\frac{1}{b_{\alpha ,|v|}}\Gamma(v,v_t,x)(1-q^{|v|-t})(1-q^{|v|-|x|})\right)q^{-\alpha t}\\
			&\gtrsim b_{\alpha,o}^{-1}\sum_{t=1}^{|x|}\sum_{v\in X}q^{\alpha(|v|-t)}\Gamma(v,v_t,x)(1-q^{|v|-t})(1-q^{|v|-|x|})\\
			&= \sum_{t=1}^{|x|}\sum_{\ell=0}^{t-1}q^{\alpha(\ell-t)}\Gamma(v_\ell,v_t,x)(1-q^{\ell-t})(1-q^{\ell-|x|})\\
			&\gtrsim \sum_{t=1}^{|x|}\sum_{\ell=0}^{t-1}q^{\alpha(\ell-t)}\simeq \sum_{t=1}^{|x|}q^{-\alpha t} q^{\alpha t}=|x|,
		\end{align*}
		where we used the fact that $\supp(\Gamma(\cdot,v_t,x))=[o,v_{t-1}]=[v_0,v_{t-1}]$ and the function is greater than or equal to $\frac{q-1}{q}$ there. 
	\end{proof}

	\begin{prop}\label{lemp=1}
		If $T_{a,b,c}$ is bounded on $L^1_\alpha$, then 
		\begin{alignat*}{2}
			&\alpha<b,\quad&& \text{when }\,c=a+b;\\
			&\alpha\leq b,\quad&&  \text{when }\,c< a+b.
		\end{alignat*}
	\end{prop}
	\begin{proof}
		From Proposition~\ref{a+b>c}, if $T_{a,b,c}$ is bounded on $L^1_\alpha$, then $c\leq a+b$. The boundedness of $T_{a,b,c}$ on $L^1_\alpha$ implies the boundedness of the adjoint operator $T_{a,b,c}^*$ on $L^\infty_\alpha$ which is given by
		\[T^*_{a,b,c} g(x)=q^{-(b-\alpha)|x|}\sum_{z\in X}K_c(x,z)g(z)q^{-(a+\alpha)|z|},\quad g\in L^\infty_\alpha.\]
		In particular, by~\eqref{harmoball}
		\begin{align*}
			T_{a,b,c}^*\mathbbm{1}_X(x)&=q^{-(b-\alpha)|x|}\sum_{z\in X}K_c(x,z)q^{-(a+\alpha)|z|}\\
			&=q^{-(b-\alpha)|x|}\sum_{n=0}^{+\infty}q^{-(a+\alpha)n}\sum_{|z|=n}K_c(x,z)\\
			&=q^{-(b-\alpha)|x|}\frac{1}{B_{c}}\sum_{n=0}^{+\infty}\#\sph{o}{n}q^{-(a+\alpha)n}=\frac{B_{a+\alpha}}{B_c}q^{-(b-\alpha)|x|},
		\end{align*}
		which belongs to $L^\infty_\alpha$ if and only if $\alpha\leq b$.

		Suppose now that $a+b=c$. We know that $\alpha\leq b$ and we want to prove that $\alpha<b$. 
		Suppose by contradiction that $\alpha=b$. For every $x\in X$ define 
		\[g_x(z)=\begin{dcases*}
			|K_c(z,x)|K_c(z,x)^{-1},&\text{if }$K_c(z,x)\neq 0$,\\
			0,&\text{otherwise}.
		\end{dcases*}\]
		Then $\|g_x\|_{L_\alpha^\infty}=1$ and
		\begin{equation}\label{condequivT*}
			T_{a,b,c}^*g_x(x)=\sum_{z\in X}|K_c(x,z)|q^{-c|z|}\gtrsim |x|,
		\end{equation}
		by Lemma~\ref{lem:nonboundp=1}. Thus $T_{a,b,c}^*$ is unbounded on $L_\alpha^\infty$ and consequently $T_{a,b,c}$ is unbounded on $L_\alpha^1$ for $\alpha=b$ and $c=a+b$.

		
	\end{proof}
	Propositions \ref{-pa<}, \ref{a+b>c}, \ref{lemp=1} show that~\ref{3.11 b} implies~\ref{3.11 c} in Theorem~\ref{3.12}.

	\subsubsection{Proof that \ref{3.11 c} implies \ref{3.11 a}}
	We start by stating a technical lemma, which will be useful both in Proposition~\ref{p>1impliesa} and Proposition~\ref{p=1impliesa}, that are devoted to prove that \ref{3.11 c} implies \ref{3.11 a} in the case $p>1$ and $p=1$, respectively.
	\begin{lem}\label{lemstimaK}
		Let $\beta,\gamma>1$. There exist $C_1,C_2>0$ depending only on $\beta$ and $\gamma$ such that
		\begin{equation*}
			\sum_{x\in X}|K_\gamma(z,x)|q^{-\beta|x|}\leq\begin{dcases*}
				C_1(1+q^{-(\beta-\gamma)|z|}),&\text{if }$\gamma\neq\beta$,\\
				C_2(1+|z|),&\text{if }$\gamma=\beta$.
			\end{dcases*}
		\end{equation*}
	\end{lem}
	\begin{proof}
		Let $z\in X$ and $\{v_j\}_{j=0}^{|z|}=[o,z]$. We start by applying~\eqref{kernelconfluente} and~\eqref{Bt} to the kernel $K_\gamma$, obtaining
		\begin{align*}
			\sum_{x\in X}&|K_\gamma(z,x)|q^{-\beta|x|}\\
			&= \sum_{x\in X}\left|\frac{1}{B_\gamma}+\frac{q^2b_{\gamma,0}^{-1}}{(q-1)^2}\sum_{j=0}^{|z\wedge x|}q^{\gamma j}\Gamma(v_j,z,x)(1-q^{j-|z|})(1-q^{j-|x|})\right|q^{-\beta|x|}\\
			&\leq \frac{B_\beta}{B_\gamma}+ \frac{q^2b_{\gamma,0}^{-1}}{(q-1)^2}\sum_{x\in X}\sum_{j=0}^{|z\wedge x|}q^{\gamma j}\left|\Gamma(v_j,z,x)\right|(1-q^{j-|z|})(1-q^{j-|x|})q^{-\beta|x|}\\
			&\leq \frac{B_\beta}{B_\gamma}+ \frac{q^2b_{\gamma,0}^{-1}}{(q-1)^2}\sum_{x\in X}q^{-\beta|x|}\sum_{j=0}^{|z\wedge x|}q^{\gamma j},
		\end{align*}
	where we used that $|\Gamma|<1$ and $|z|,\,|x|>j$. Now observe that, for every $x\in X$ and $0\leq\ell<|z|$, $|z\wedge x|=\ell$ is equivalent to $x\in \sect{v_\ell}\setminus\sect{v_{\ell+1}}$ and $z\wedge x=z$ if and only if $x\in\sect{z}$. Furthermore, we have that \begin{equation*}
		|\sph{o}{m}\cap\sect{v_\ell}\setminus\sect{v_{\ell+1}}|=\begin{dcases*}
			0,&\text{if }$m<\ell$;\\
			1,&\text{if }$m=\ell$;\\
			(q-1)q^{m-\ell-1},&\text{if }$m>\ell$.
		\end{dcases*}
	\end{equation*}
	Hence, we have
	\begin{align*}
		\sum_{x\in X}q^{-\beta|x|}\sum_{j=0}^{|z\wedge x|}q^{\gamma j}&=\sum_{\ell=0}^{|z|-1}\sum_{x\in \sect{v_\ell}\setminus\sect{v_{\ell+1}}}q^{-\beta|x|}\sum_{j=0}^{\ell}q^{\gamma j}+\sum_{x\in \sect{z}}q^{-\beta|x|}\sum_{j=0}^{|z|}q^{\gamma j}\\
		&\simeq\sum_{\ell=0}^{|z|-1}\left(\sum_{j=0}^{\ell}q^{\gamma j}\right)\sum_{m=\ell}^{+\infty}q^{m-\ell}q^{-\beta m}+\sum_{n=|z|}^{+\infty}q^{n-|z|}q^{-\beta n}\sum_{j=0}^{|z|}q^{\gamma j}\\
		&\simeq \sum_{\ell=0}^{|z|}q^{(\gamma-1) \ell}\sum_{m=\ell}^{+\infty}q^{(1-\beta)m}\simeq \sum_{\ell=0}^{|z|}q^{(\gamma-\beta)\ell},
	\end{align*}
	where we used that $|T_z\cap\sph{o}{n}|=q^{n-|z|}$ when $n\geq|z|$ and $\beta>1$.
	This proves that there exist $C_1,C_2>0$ (depending on $\gamma$ and $\beta$) such that the thesis holds true.
	\end{proof}
	
	\begin{prop}\label{p>1impliesa}
		Let $1<p<\infty$. If $a+b\geq c>1$ and $-pa<\alpha-1<p(b-1)$, then $S_{a,b,c}$ is bounded on $L^p_\alpha$.
	\end{prop}
	\begin{proof}
		We set
		\[H(z,x)=|K_c(z,x)|q^{-a|z|}q^{-(b-\alpha)|x|},\]
		so that the operator $S_{a,b,c}$ becomes
		\[ S_{a,b,c}f(z)=\sum_{x\in X}H(z,x)f(x)q^{-\alpha|x|}.\]
		Our purpose is to apply Schur's test (see Theorem~3.6 in~\cite{zhuoperator}) to the integral operator with positive kernel $H\colon X\times X\to [0,+\infty)$. To do so, we have to show that there exists a positive function $h$ on $X$ such that 
		\begin{equation}\label{schur}
			\sum_{z\in X}H(z,x)h(z)^{p'}q^{-\alpha|z|}\lesssim h(x)^{p'},\qquad \sum_{x\in X}H(z,x)h(x)^{p}q^{-\alpha|x|}\lesssim h(z)^{p}.
		\end{equation}
		Observe that the two inequalities assumed for $\alpha$ are equivalent to 
		\[-\frac{a+\alpha-1}{p}<\frac{a}{p'},\qquad-\frac{b-1}{p'}<\frac{b-\alpha}{p}.\]
		Hence, since $a+b>1$, it is possible to choose an element 
		\begin{equation}\label{gamma}
			\gamma\in\left(-\frac{b-1}{p'},\frac{a}{p'}\right)\cap\left(-\frac{a+\alpha-1}{p},\frac{b-\alpha}{p}\right)\neq\emptyset.
		\end{equation}
		We want to show that $h(x)=q^{-\gamma|x|}$ satisfies~\eqref{schur}. Let $z\in X$. We suppose $\gamma\neq\frac{c-b}{p'}$. We can apply Lemma~\ref{lemstimaK} since $b+\gamma p'>1$ by ~\eqref{gamma}, obtaining
		\begin{align*}
			\sum_{x\in X}H(z,x)h(x)^{p'}q^{-\alpha |x|}&=q^{-a|z|} \sum_{x\in X}|K_c(z,x)|q^{-(b+\gamma p')|x|}\\
			&\lesssim 
			q^{-a |z|} (1+q^{-(b+\gamma p'-c)|z|})
			\\
			&\lesssim q^{-\gamma p'|z|}=h(z)^{p'},
		\end{align*}
		where we used $a+b-c\geq 0$ and $a>\gamma p'$. Similarly, when $\gamma=\frac{c-b}{p'}$ we can apply again Lemma~\ref{lemstimaK} and conclude by using $a>\gamma p'$.
		On the other hand, we have that if $\gamma\neq\frac{c-a-\alpha}{p}$, by $a+\gamma p+\alpha>0$ and by Lemma~\ref{lemstimaK},
		\begin{align*}
			\sum_{z\in X}H(z,x)h(z)^{p}q^{-\alpha |z|}&=q^{-(b-\alpha)|x|}\sum_{z\in X}|K_c(z,x)|q^{-(a+\gamma p+\alpha )|z|}\\
			&\lesssim
			q^{-(b-\alpha)|x|} (1+q^{-(a+\gamma p+\alpha-c)|z|})
			\\
			&\lesssim q^{-\gamma p|z|}=h(z)^{p},
		\end{align*}
		since $a+b\geq c$ and, by~\eqref{gamma}, $b-\alpha>\gamma p$. Similarly when $\gamma=\frac{c-a-\alpha}{p}$.
		
		In conclusion, \eqref{schur} holds and by Schur's test the operator $S_{a,b,c}$ is bounded on $L^p_\alpha(X)$. 
	\end{proof}
	Notice that Proposition \ref{p>1impliesa} shows that~\ref{3.11 c} implies~\ref{3.11 a} in Theorem~\ref{3.11}.

	\begin{prop}\label{p=1impliesa}
		If $a+b\geq c$ and
		\begin{alignat*}{2}
			&-a<\alpha-1<b-1,\quad&& \text{when }\,c=a+b;\\
			&-a<\alpha-1\leq b-1,\quad&&  \text{when }\,c< a+b,
		\end{alignat*}
		then $S_{a,b,c}$ is bounded on $L^1_\alpha$.
	\end{prop}
	\begin{proof}
		Let $f\in L^1_\alpha$. We suppose $c\neq a+\alpha$ and we observe that, since $a+\alpha>1$, by Lemma~\ref{lemstimaK}  
		\begin{align*}
			\|S_{a,b,c}f\|_{L^1_\alpha}&=\sum_{z\in X}\left| \sum_{x\in X} |K_c(z,x)|f(x)q^{-b|x|}\right|q^{-(a+\alpha)|z|}\\
			&\leq
			\sum_{x\in X}|f(x)|q^{-b|x|} \sum_{z\in X} |K_c(z,x)|q^{-(a+\alpha)|z|}\\
			&\lesssim 
			\sum_{x\in X}|f(x)|q^{-b|x|} (1+q^{-(a+\alpha-c)|x|})
			\\
			&\lesssim \sum_{x\in X}|f(x)|q^{-\alpha|x|}=\|f\|_{L^1_\alpha},
		\end{align*}
		where we used the fact that $a+b-c\geq 0$ and $b\geq\alpha$. The case $c=a+\alpha$ follows similarly using again Lemma~\ref{lemstimaK} and $b>\alpha$. 
		Hence, $S_{a,b,c}$ is bounded on $L^1_\alpha$.
	\end{proof}
	
	Proposition \ref{p=1impliesa} shows that~\ref{3.12 c} implies~\ref{3.12 a} in Theorem~\ref{3.12}.

	\section{Calder\'on-Zygmund decomposition}\label{seccz}
	
	In this section, we discuss a Calder\'on-Zygmund decomposition of functions in $L_\alpha^1$ and we formulate the integral H\"ormander's condition for kernels on the tree which guarantees the weak type (1,1) boundedness of integral operators which are bounded on $L_\alpha^2$ . As byproduct, we have that $P_\alpha$ is of weak type (1,1) for every $\alpha>1$. 
	
	
	
	By Proposition \ref{nondoubling} the measure metric space $(X,d,\mu_{\alpha})$ is nondoubling. We now introduce the Gromov distance $\rho$, see~\cite{arcozziRSW},~\cite{Gromov}, and show that the measure metric space $(X,\rho,\mu_{\alpha})$ is doubling. For every $u,v\in X$ define
	\begin{equation}
		\rho(v,u)=\begin{cases*}
			0,& if $u=v$; \\
			e^{-|v\wedge u|},& if $v\neq u$.
		\end{cases*}
	\end{equation}
	For every $v\in X$, observe that if $u\in X\setminus\{v\}$ then $\rho(v,u)=e^{-|v\wedge u|}\in [e^{-|v|},1]$ and $|v\wedge u|=-\log(\rho(v,u))$, that is \[u\in\sect{p^{|v|+\log(\rho(v,u))}(v)}\setminus\sect{p^{|v|+\log(\rho(v,u))-1}(v)}.\]
	Thus, the nontrivial balls with respect to $\rho$ centred at $v$ are sectors of the tree. More in general, we have
	\begin{equation}\label{balle}
		B_\rho(v,r):=\{ u\in X\colon\rho(v,u)<r\}=\begin{cases*}
			\{v\},&if $0<r\leq e^{-|v|}$,\\
			T_{p^{|v|+\lfloor\log r\rfloor}(v)},&if $e^{-|v|}<r\leq 1$,\\
			X,&if $r>1$.
		\end{cases*}
	\end{equation}
	Observe that in the special case $v=o$ we have that $B_\rho(o,r)=\{o\}$ if $0<r\leq 1$ and $B_\rho(o,r)=X$ for every $r>1$.
	Hence every vertex $v$ is the center of exactly $|v|+2$ balls.  
	
	\begin{prop}\label{doubling}
		For every $\alpha>1$ the measure metric space $(X,\rho,\mu_\alpha)$ is globally doubling with doubling constant \[D_\alpha=\max\left\{q^\alpha+1,\frac{q^\alpha+1}{q^\alpha-q}\right\},\]
		that is 
		\[\mu_\alpha(B_\rho(v,2r))\leq D_\alpha\mu_\alpha(B_\rho(v,r)),\qquad v\in X,\,r>0.\]
	\end{prop}
	\begin{proof}
		Let $\alpha>1$. We start by observing that for every $u\in X\setminus\{o\}$
		\begin{equation}\label{misurasettore}
			\mu_\alpha(\sect{u})=\sum_{\ell=0}^{+\infty}q^{\ell}q^{-\alpha(\ell+|u|)}=q^{-\alpha|u|}\frac{1}{1-q^{1-\alpha}}.
		\end{equation}
		Let $0<r\leq 1$. Observe that if $\{x\}:=x-\lfloor x\rfloor\in[0,1)$, then
		\begin{equation*}
			\lfloor \log(2r)\rfloor=\begin{cases*}
				\lfloor \log r\rfloor,&if $0\leq \{\log r\}<1-\log2$,\\
				1+\lfloor \log r\rfloor,&if $1-\log2\leq \{\log r\}<1$.
			\end{cases*}
		\end{equation*}
		Hence whenever $B_\rho(v,r)=\{v\}$ we have that $B_\rho(v,2r)\in \{\{v\},\sect{v}\}$, and if $B_\rho(v,r)=\sect{u}$ for some $u\in X\setminus\{o\}$ then $B_\rho(v,2r)\in\{\sect{u},\sect{p(u)}\}$.

		If $v\in X\setminus\{o\}$, then 
		\begin{equation}\label{rapportosettorevertice}
			\frac{\mu_\alpha(\sect{v})}{\mu_\alpha(\{v\})}=\frac{q^{-\alpha|v|}(1-q^{1-\alpha})^{-1}}{q^{-\alpha|v|}}=\frac{1}{1-q^{1-\alpha}}.
		\end{equation}
		If $|v|>1$, then 
		\begin{equation}\label{rapportosettori}
			\frac{\mu_\alpha(\sect{p(v)})}{\mu_\alpha(\sect{v})}=\frac{q^{-\alpha(|v|-1)}(1-q^{1-\alpha})^{-1}}{q^{-\alpha|v|}(1-q^{1-\alpha})^{-1}}=q^\alpha.
		\end{equation}
		If $|v|=1$, then 
		\begin{equation}\label{rapportosettoretutto}
			\frac{\mu_\alpha(X)}{\mu_\alpha(\sect{v})}=\frac{(1+q^{-\alpha})(1-q^{1-\alpha})^{-1}}{q^{-\alpha}(1-q^{1-\alpha})^{-1}}=q^\alpha+1.
		\end{equation}
		Finally, we consider the case $v=o$. In this case, it is sufficient to check that, by~\eqref{misuratotalealpha}
		\[\frac{\mu_\alpha(X)}{\mu_\alpha(\{o\})}=\frac{1+q^{-\alpha}}{1-q^{1-\alpha}}=\frac{q^\alpha+1}{q^\alpha-q}.\]
		Hence $(X,\rho,\mu_\alpha)$ is doubling with constant $D_\alpha=\max\{q^\alpha+1,(q^\alpha+1)/(q^\alpha-q)\}$.
	\end{proof}
	
	As a consequence of Proposition \ref{doubling} in this setting one can develop a classical Calder\'on-Zygmund theory using the balls of the Gromov metric, i.e.~using sectors (see \cite{coifweiss},~\cite{SteinHA}). Our argument is inspired by~\cite[Theorem 1.1]{denghuang}, where a similar construction is developed in the setting of the hyperbolic disk (see also the Whitney decomposition in~\cite{ArcRocSaw}).
	 Since it is not difficult to construct an explicit decomposition algorithm for sectors and then describe the associated Calder\'on-Zygmund decomposition of integrable functions, we think that it is worthwhile discussing this construction in detail, as we do next.

	We start with a preliminary geometrical result that allows us to obtain an infinite family of partitions of a sector. In particular, the partition at a given scale is a refinement of the partition at the previous scale, and the measure of a partitioning set is comparable with the measure of the set which contains it in the previous partition.
	\begin{lem}\label{lemdecomp}
		Let $v\in X\setminus\{o\}$. For every $m\in\N$, there exists $I_m\in\N$ and sets $Q_{k,m}\subseteq\sect{v}$ for every $k\in\cI_m:=\{0,\dots,I_m\}$ such that
		\begin{enumerate}[label=(\roman*)]
			\item\label{i} $Q_{k,m}\cap Q_{k',m}=\emptyset$ for every $k\neq k'$;
			\item\label{ii} the sector $\sect{v}$ is the disjoint union of the sets $Q_{k,m}$, $k\in\cI_m$;
			\item\label{iii} the partition at scale $m>0$ is a refinement of the partition at scale $m-1$, that is, for every $k'\in\cI_{m-1}$ there exists $\cI_{m,k'}\subseteq\cI_m$ such that \[Q_{k',m-1}=\bigsqcup_{k\in\cI_{m,k'}}Q_{k,m};\]
			\item\label{iiii} for every $k\in\cI_m$ and $k'\in\cI_{m-1}$ for which $Q_{k,m}\subseteq Q_{k',m-1}$, we have
			\[\mu_\alpha(Q_{k,m})\leq\mu_\alpha(Q_{k',m-1})\leq D_{\alpha}\mu_\alpha(Q_{k,m}).\]
		\end{enumerate}
	\end{lem}
	Observe that in (iv) the constant $D_\alpha$ can be replaced by $\max\{q^\alpha,(1-q^{1-\alpha})^{-1}\}$, because we focus only on $\sect{v}$. 
	\begin{proof}
		For every $m\in\N$ we set
		\[I_m=\frac{q^{m+1}-q}{q-1}.\]
		We label the vertices of $\sect{v}$ in such a way that $v_0=v$ and $s(v_k)=\{v_{qk+\ell}\colon \ell\in\{1,\dots,q\}\}$ for every $k\in\N$. Since $\cI_0=\{0\}$ it is sufficient to set $Q_{0,0}=\sect{v}$. Then for every $m\in \N\setminus\{0\}$ we set
		\begin{alignat*}{2}
			Q_{k,m}&:=\{v_k\},\qquad&&\text{if }k\in\cI_{m-1},\\ 
			Q_{k,m}&:=\sect{v_k},&&\text{if }k\in\cI_m\setminus\cI_{m-1}. 
		\end{alignat*}
		In this way, \ref{i}, \ref{ii}, and \ref{iii} easily follow by construction.   Finally,~\ref{iiii} follows from~\eqref{rapportosettorevertice},~\eqref{rapportosettori},~\eqref{rapportosettoretutto}, and the fact that 
		\begin{equation}
			Q_{k',m-1}\in\begin{dcases*}
				\{\sect{v},\sect{p(v)} \},&if $Q_{k,m}=\sect{v}$;\\
				\{\{v\},\sect{v} \},&if $Q_{k,m}=\{v\}$.
			\end{dcases*}
		\end{equation}
	\end{proof}
	The previous result leads to a Calder\'on-Zygmund decomposition for integrable functions on the tree at a level $t\in\R^+$ sufficiently large w.r.t~the $L_\alpha^1$-norm of the function.
	\begin{prop}
		\label{caldzyg}
		Let 
		$f\in L_\alpha^1$ and $t>\|f\|_{L^1_\alpha}/\mu_\alpha(X)$. There exist two families $\cQ$ and $\cF$ of disjoint sets of the form $Q_{k,m}$ such that, if we denote by $\Omega$ and $F$ the disjoint union of all the sets in $\cQ$ and $\cF$, respectively, 
		the following properties hold:
		\begin{enumerate}
			\item $X=\Omega\sqcup F$;
			\item $|f(z)|\leq t$ for every $z\in F$;
			\item there exist $g,\,b\colon X\to\C$ and $C>0$ such that $f=g+b$, $\supp\, b\subseteq \Omega$, and $\|g\|_{L_\alpha^2}^2\lesssim t\|f\|_{L_\alpha^1}$. Moreover, if we set $b_Q=b\mathbbm{1}_Q$ for every $Q\in\cQ$, then
			\[\sum_{z\in Q} b_Q(z)q^{-\alpha|z|}=0,\qquad \sum_{Q\in\cQ}\|b_Q\|_{L_\alpha^1}\leq C\|f\|_{L_\alpha^1},\qquad Q\in\cQ.\]
			%
			%
		\end{enumerate}
	\end{prop}
	\begin{proof}
		For every $v\in\sph{o}{1}$ we consider the decomposition of the sector $T_v$ given by Lemma~\ref{lemdecomp}. We define two families of subsets $\mathcal{Q}_v$ and $\cF_v$ following the steps below. Starting from $Q_{k,m}=Q_{0,0}=\sect{v}$:
		\begin{enumerate}
			\item[1)] if 
			\[	\frac 1{\mu_\alpha(Q_{k,m})}\sum_{z\in Q_{k,m}}|f(z)|q^{-\alpha|z|}> t,\]
			then we put $Q_{k,m}\in\mathcal{Q}_v$ and we stop. Otherwise,
			\item[2a)] if $\# Q_{k,m}=1$ then $Q_{k,m}\in \cF_v$ and we stop;
			\item[2b)] if $\# Q_{k,m}>1$ then for each set in the family \[Q_{k,m+1}\cup\{Q_{kq+j,m+1}\colon j\in{1,\dots q}\}\]
			we repeat the procedure, starting from 1).
		\end{enumerate}
		We define
		\[\cQ:=\begin{dcases*}
			\bigsqcup_{v\in\sph{o}{1}}\cQ_v,&\text{if }$|f(o)|\leq t$;\\
			\{o\}\cup \bigsqcup_{v\in\sph{o}{1}}\cQ_v,&\text{otherwise,}
		\end{dcases*} \quad \cF:=\begin{dcases*}
			\{o\}\cup\bigsqcup_{v\in\sph{o}{1}}\cF_v,&\text{if }$|f(o)|\leq t$;\\
			\bigsqcup_{v\in\sph{o}{1}}\cF_v,&\text{otherwise.}
		\end{dcases*}\]
		We denote by $\Omega$ and $F$ the (disjoint) union of all the subsets in $\cQ$ and $\cF$, respectively. The sets $\Omega$ and $F$ clearly satisfy (i) and (ii).
		We prove that, for every $Q\in\cQ$,
		\begin{equation}\label{tipot}
			t<  \frac 1{\mu_\alpha(Q)}\sum_{z\in Q}|f(z)|q^{-\alpha|z|}\leq C_\alpha t,\qquad Q\in\cQ.
		\end{equation}
		For every $Q\in \cQ$ we put\[\tilde{Q}=\begin{dcases*}
			X,&if $Q=\{o\}$ or $Q=Q_{0,0}\in\cQ_v$, $v\in\sph{o}{1}$\\
			Q_{k',m-1}&if $Q=Q_{k,m}\in\cQ_v$, $m>0$, $v\in\sph{o}{1}$,
		\end{dcases*}\]
		where $k'$ is defined in (iv) of Lemma~\ref{lemdecomp}. Observe that $\tilde{Q}\not\in\cQ$ and that, by Proposition~\ref{doubling}, $\mu_\alpha(\tilde{Q})\leq C_\alpha\mu_\alpha(Q)$. Then we have that
		\[\frac{1}{\mu_\alpha(Q)}\sum_{z\in Q}|f(z)|q^{-\alpha|z|}\leq \frac{\mu_\alpha(\tilde Q)}{\mu_\alpha(Q)}\frac{1}{\mu_\alpha(\tilde Q)}\sum_{z\in \tilde Q}|f(z)|q^{-\alpha|z|}\leq C_\alpha t,\]
		which gives~\eqref{tipot}.
		It is easy to see that 
		\begin{equation}\label{Omegaminore}
			\mu_\alpha(\Omega)\leq \frac{1}{t}\sum_{Q\in\cQ}\sum_{x\in Q}\frac{1}{\mu_\alpha(Q)}\left(\sum_{z\in Q}|f(z)|q^{-\alpha|z|}\right)q^{-\alpha|x|}\leq \frac{\|f\|_{L_\alpha^1}}{t}.
		\end{equation}
		
		We now define $b=f-g$, where
		\begin{equation*}
			g(z)=\begin{dcases*}
				f(z),&$z\in F$;\\
				\frac1{\mu_\alpha(Q)}\sum_{x\in Q}f(x)q^{-\alpha|x|},&$z\in Q$.
			\end{dcases*}
		\end{equation*}
		It is obvious that $\supp\,b\subseteq\Omega$. We show next that $\|g\|_{L^2_\alpha}^2\leq (1+C_\alpha^2)t\|f\|_{L_\alpha^1}$. Indeed, by~\eqref{tipot},
		\begin{align*}
			\|g\|_{L_{\alpha}^2}^2&=\sum_{z\in F}|g(z)|^2q^{-\alpha|z|}+\sum_{z\in \Omega}|g(z)|^2q^{-\alpha|z|}\\
			&=\sum_{z\in F}|f(z)|^2q^{-\alpha|z|}+\sum_{Q\in\cQ}\sum_{z\in Q}\left|\frac1{\mu_\alpha(Q)}\sum_{x\in Q}f(x)q^{-\alpha|x|}\right|^2q^{-\alpha|z|}\\
			&\leq\sum_{z\in F}t|f(z)|q^{-\alpha|z|}+\mu_\alpha(\Omega) C_\alpha^2t^2\leq
			(1+C_\alpha^2)t\|f\|_{L^1_\alpha}<+\infty,
		\end{align*}
		where we used~\eqref{Omegaminore}.
		The fact that $b_Q=b\mathbbm{1}_Q$, $Q\in\cQ$, has vanishing mean on $Q$ follows by construction. Furthermore, since $|b(z)|\leq |f(z)|+|g(z)|$ we have
		\begin{align*}
			\sum_{Q\in \cQ}\sum_{z\in Q}|b_Q(z)|q^{-\alpha|z|}&\leq \sum_{z\in \Omega}|f(z)|q^{-\alpha|z|}+ \sum_{Q\in \cQ}\sum_{z\in Q}|g(z)|q^{-\alpha|z|}\\
			&\leq \|f\|_{L_\alpha^1}+\mu_\alpha(\Omega)C_\alpha t\lesssim \|f\|_{L_\alpha^1},
		\end{align*}
		by~\eqref{Omegaminore}.
	\end{proof}

	In the doubling measure metric space $(X,\rho,\mu_\alpha)$, the standard integral H\"ormander's condition (see~\cite{hormander} and formula~(10) Ch.I in~\cite{SteinHA}) for a kernel $K\colon X\times X\to \C$ is 
	\begin{equation*}
		\sup_{v\in X,r>0}\sup_{x,y\in B_\rho(v,r)}\int_{X\setminus B_\rho(v,2r)}|K(z,x)-K(z,y)|\mu_\alpha(z)<+\infty.
	\end{equation*}
	Thanks to the shape of the balls, see~\eqref{balle}, it is equivalent to
	\begin{equation}\label{hormander}
		\sup_{v\in X\setminus\{o\}}\sup_{x,y\in \sect{v}}\sum_{z\in X\setminus \sect{v}}|K(z,x)-K(z,y)|q^{-\alpha|z|}<+\infty.
	\end{equation}
	Notice that this is precisely what is proved to hold in Proposition~\ref{prophormalpha} for the Bergman kernel $K_\alpha$.
	We then have the following boundedness result for integral operators (see Theorem 3 Ch.I \cite{SteinHA}).

	\begin{thm}\label{thmweak}
		Fix $\alpha>1$ and let $K\colon X\times X\to \C$  be a kernel  satisfying the H\"ormander's condition~\eqref{hormander} with respect to $\mu_\alpha$. If the integral operator defined on functions $f\in L_\alpha^2$ by
		\[	\cK f(z)=\sum_{x\in X}K(z,x)f(x)q^{-\alpha|x|}\]
		is bounded on $L_\alpha^2$, then $\cK$
		is of weak type (1,1).
		Furthermore, $\cK$ admits a bounded extension $\cK$ on $L_\alpha^p$, for every $1<p< 2$.
	\end{thm}

	The following result is obtained as byproduct of Proposition~\ref{prophormalpha} and Theorem~\ref{thmweak}. It is a discrete counterpart of the result for (unweighted and holomorphic) Bergman spaces on the hyperbolic disk obtained in~\cite{denghuang}.
	\begin{cor}\label{weakone}
		The Bergman projector $P_\alpha$ is of weak type $(1,1)$, for every $\alpha>1$.
	\end{cor}

	\subsubsection*{Acknowledgment}
	The authors are grateful to Marco~M.~Peloso for 
	suggesting references on the weak type (1,1) boundedness of the Bergman projector in
	the case of the hyperbolic disk, and to Matteo~Levi for useful comments.
	
	\bibliographystyle{plain}
	
\end{document}